%% file: interp2.tex
\documentclass[11pt]{amsart}
\usepackage{amsmath,amssymb}

\usepackage{inputenc}
\usepackage{bm}
\usepackage{xcolor}
\usepackage{newtxmath}
\usepackage{mathrsfs}
\usepackage{verbatim}
\usepackage{float}
\usepackage{setspace}
\usepackage{amsthm}
\usepackage{geometry}
\usepackage{color}
\usepackage{cite}
\usepackage{enumitem}
\usepackage{tikz}
\usetikzlibrary{backgrounds}
\usetikzlibrary{arrows}
\usetikzlibrary{shapes,shapes.geometric,shapes.misc}
\usepackage{hyperref}
\hypersetup{
	colorlinks=true,
	linkcolor=cyan!50!blue,
	filecolor=blue,      
	urlcolor=red,
	citecolor=green,
}
\usepackage{cleveref}
\usepackage[all]{xy}
\allowdisplaybreaks[4]
\usepackage{geometry}
 \newtheorem{theorem}{Theorem}[section]
 \newtheorem{lemma}{Lemma}[section]
 \newtheorem{proposition}{Proposition}[section]
\theoremstyle{remark}
 \newtheorem{remark}{Remark}[section]

\usepackage{appendix}
\usepackage{chemarrow}
\usepackage{extarrows}
\usepackage{mathtools}

\newcommand{\cM}{\mathcal{M}}
\newcommand{\bS}{\mathbb{S}}
\newcommand{\bT}{\mathbb{T}}

\newcommand{\R}{\mathbb{R}}

\newcommand{\id}{{\rm{id}}}

\newcommand{\hess}{\nabla^2}

\newcommand{\bx}{\mathbf x}
\newcommand{\be}{\mathbf e}

\DeclareMathOperator{\Span}{span}

\numberwithin{equation}{section}
\numberwithin{table}{section}
\numberwithin{figure}{section}

\DeclareMathOperator{\grad}{\nabla}
\DeclareMathOperator{\curl}{curl}
\DeclareMathOperator{\divergence}{div}
\renewcommand{\div}{\divergence}
\DeclareMathOperator{\sym}{sym}
\DeclareMathOperator{\dev}{dev}  
\DeclareMathOperator{\rot}{rot}

\newcommand{\cQ}{\mathcal Q}
\newcommand{\cR}{\mathcal R}

\makeatletter
\@addtoreset{equation}{section}
\makeatother

\begin{document}
\title[Commuting Projection Operators for Discrete Gradgrad Complexes]{Local Bounded Commuting Projection Operators for Discrete Gradgrad Complexes}
\author{Jun Hu}
\address{LMAM and School of Mathematical Sciences, Peking University, Beijing 100871,
P. R. China.}
\email{hujun@math.pku.edu.cn}
\author{Yizhou Liang}
\address{Institute of Mathematics, University of Augsburg, Universit\"{a}tsstraße 12A, 86159 Augsburg, Germany}
\email{yizhou.liang@uni-a.de}
\author{Ting Lin}
\address{School of Mathematical Sciences, Peking University, Beijing 100871,
P. R. China.}
\email{lintingsms@pku.edu.cn}
\thanks{The first author was supported by the National Natural Science Foundation of China grants NSFC 12288101. The second author was supported by Humbodlt Research Fellowship for Postdocs.}

\begin{abstract}
    This paper discusses the construction of local bounded commuting projections for discrete subcomplexes of the gradgrad complexes in two and three dimensions, which play an important role in the finite element theory of elasticity (2D) and general relativity (3D). The construction first extends the local bounded commuting projections to the discrete de Rham complexes to other discrete complexes. Moreover, the argument also provides a guidance in the design of new discrete gradgrad complexes.
\end{abstract}
\maketitle
\section{Introduction}

This paper focuses on the construction of local bounded commuting projections for discrete subcomplexes of the gradgrad complexes in two and three dimensions. The gradgrad complexes play an important role in the finite element exterior calculus, as it provides a systematic understanding to the linearized Einstein--Bianchi equation\cite{2015Quenneville}. For two dimensions, it can be regarded as a rotation of the elasticity complex. The continuous gradgrad complex in two dimension reads as \cite{2002ArnoldWinther}:
\begin{equation*}
	P_1 \xrightarrow{\subset} H^2(\Omega) \xrightarrow[]{\operatorname{grad}\operatorname{grad}} H(\rot, \Omega; \bS) \xrightarrow[]{\operatorname{rot}} L^2(\Omega; \R^2)\xrightarrow{}0,
\end{equation*} 
and in three dimensions\cite{2020PaulyZulehner} reads as:
\begin{equation*}
	P_1 \xrightarrow{\subset} H^2(\Omega) \xrightarrow[]{\operatorname{grad}\operatorname{grad}} H(\curl, \Omega; \bS) \xrightarrow[]{\operatorname{curl}} H(\div, \Omega;\mathbb T) \xrightarrow[]{\operatorname{div}} L^2(\mathbb R^3)\xrightarrow{}0.
\end{equation*} 
Here $\mathbb{S}$ denotes the spaces of symmetric matrices in two and three dimensions, and $\mathbb{T}$ denotes the traceless matrices in three dimensions, the operators $\operatorname{rot}$, $\operatorname{div}$ and $\operatorname{curl}$ are applied by row. The exactness of these complexes is ensured when the domain $\Omega$ is contractible and Lipschitz \cite{2002ArnoldWinther,2020PaulyZulehner}. Recently, the first finite element sub-complex of the gradgrad complex in three dimensions was constructed \cite{2021HuLiang}, and the two-dimensional case was proposed using the Bernstein--Gelfand--Gelfand construction \cite{2018ChristiansenHuHu}.

Bounded commuting projections from the Hilbert variant of the de Rham complex to a finite dimensional subcomplex have been a primary instrument in the finite element exterior calculus\cite{2006ArnoldFalkWinther,2010ArnoldFalkWinther,2018Arnold}. See \cite{2014FakWinther,2021ArnoldGuzman} for the standard finite element de Rham complexes, and \cite{2023HuLiangLin} for the nonstandard ones. However, the construction of local bounded commuting projections of other complexes, e.g., the gradgrad complex, the divdiv complex and the elasticity complex, is still a challenging problem.

In this paper, we extend the framework introduced in \cite{2023HuLiangLin} to construct local bounded commuting projections for the gradgrad complexes in $\mathbb{R}^{D}(D=2,3)$. Specifically, we construct these projections from the gradgrad complex to the finite element gradgrad complexes introduced in \cite{2021HuLiang,2018ChristiansenHuHu}. We also introduce another example of a finite element complex on Clough--Tocher split, cf. \cite{2022ChristiansenHu}. Our approach involves two parts: the first part is related to the skeletal complexes, inspired by the techniques of Arnold and Guzm\'{a}n \cite{2021ArnoldGuzman}; the second part is based on the harmonic inner product with bubble function complexes on edges and faces, as in \cite{2023HuLiangLin}. Notably, there exist two different bubble function complexes on edges and faces in two and three dimensions, respectively. In summary, this paper extends the scope of these projections to non-standard complexes and provides new insights into their construction for the gradgrad complexes in $\mathbb{R}^{D}(D=2,3)$.

The rest of the paper is organized as follows. We first introduce the finite element gradgrad complexes in two and three dimensions, and propose the main result in \Cref{sec:fe}. Then we prove the main result in three dimensions (i.e., \Cref{thm:main-3d}) in \Cref{sec:3d}. Finally, we show in \Cref{sec:2d} that the argument can be also extended to two dimensions, with another example of a discrete gradgrad complex in the Clough--Tocher split.
\section{Finite Element Gradgrad Complexes in 2D and 3D}
\label{sec:fe}
Given any domain $\omega$ in $\mathbb R^2$ or $\mathbb R^3$, 
denote by $(u,v)_{\omega}$ the standard inner product $\int_{\omega} u\cdot v$ for scalar, vector-valued or matrix-valued functions $u,v \in L^2(\omega)$. Suppose $X \subset L^2(\omega)$, $u \in L^2(\omega)$, the notation $u \perp X$ means $(u, v)_{\omega} = 0, \forall v \in X$. Additionally, $u \perp \mathbb R$ means $(u,1)_{\omega} = 0$. 
For $A, B \subset L^2(\omega)$, the quotient space $A/B$ is specified as $A/B := \{u \in A : (u, v)_{\omega} =0,\,\, \forall v \in B\}.$ 

\subsection{Finite element gradgrad complex in two dimensions}
\label{sec:dof-2d}

For two dimensions, suppose that a contractible polygonal domain $\Omega$ and a simplicial triangulation $\mathcal T$ are given. Denote by $\mathsf V$ the set of vertices, $\mathsf E$ the set of edges, $\mathsf F$ the set of faces, and $\mathsf S = \mathsf V\cup \mathsf E \cup \mathsf F$ the set of all simplices. The notations $\bm x$, $\bm e$, $\bm f$ are used to represent any vertex, edge, and face of the mesh $\mathcal T$, respectively.
Given a subsimplex $\sigma$ (vertex, edge and face) of $\mathcal T$, define the local patch $\omega_{\sigma}:= \cup\{\bm f: \sigma \subset \overline{\bm f}\}$, and the extended local patches sequentially:
$$\omega_{\sigma}^{[m]} := \cup_{\bm f}\{ \bm f : \overline{\bm f} \cap \overline{\omega_{\sigma}^{[m-1]}} \neq \emptyset \}, \text{ with } \omega_{\sigma}^{[0]} = \omega_{\sigma}. $$ Here $\overline{\omega}$ is the closure of the domain $\omega$, see \Cref{fig:omega} for an illustration. Moreover, define the extended patch $\omega_{\sigma}^{h}: =\cup \{\omega_{\bm{x}}:\bm{x}\text{ is a vertex of } \sigma\}$. Then it holds that
$
	\omega_{\sigma}\subseteq \omega_{\sigma}^{h} \subseteq\omega_{\sigma}^{[1]}.
$
\begin{figure}[htbp]
	\centering
	\input{tikzfile/omega_x.tikz}
	\input{tikzfile/omega_e.tikz}
	\input{tikzfile/omega_f.tikz}
	\caption{The (extended) local patch, $\omega_{\bm x}^{[1]}$, $\omega_{\bm e}^{[1]}$ and $\omega_{\bm f}^{[1]}$. The colored regions represent $\omega_{\bm x}$, $\omega_{\bm e}$ and $\omega_{\bm f}$ respectively.}
	\label{fig:omega}
\end{figure}
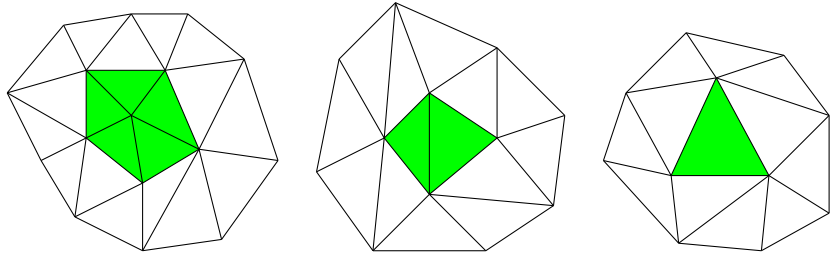

Let $\mathbf{x} = [x_1, x_2]^T$, $\mathbf{e}_1 = [1,0]^T$ and $\mathbf{e}_2 = [0,1]^T$ be vectors in two dimensions. 
Let $P_k(\bm f)$ be the space of polynomials with degree not greater than $k \ge 0$, let $RT = RT(\bm f):= \{ \bm \alpha + \beta \bx: \bm \alpha \in \mathbb R^2, \beta \in \mathbb R\}$ be the lowest order Raviart--Thomas shape function space; for a given edge $\bm e \in \mathsf E$, let $\bm n$ be its unit normal vector, $\bm t$ be its unit tangential vector, $\lambda_0,\lambda_1$ be its barycenter coordinates. For a given face $\bm f \in \mathsf F$, let $\lambda_0,\lambda_1, \lambda_2$ be its barycenter coordinates.

The two-dimensional discrete gradgrad complex, starting with the Argyris $H^2$ conforming element space, will be introduced below (cf. \cite{2018ChristiansenHuHu}),
\begin{equation}\label{eq:2D:hess} 
    P_1 \xrightarrow{\subset} U_h \xrightarrow[]{\operatorname{\hess}} \bm {\Sigma}_h \xrightarrow[]{\operatorname{rot}} \bm Q_h\xrightarrow{}0,
\end{equation} 
which is a discretization of the following continuous gradgrad complex,
\begin{equation}\label{eq:2D:hesscon} 
     P_1 \xrightarrow{\subset} H^2(\Omega) \xrightarrow[]{\operatorname{\hess}} H(\rot, \Omega; \bS) \xrightarrow[]{\operatorname{rot}} L^2(\Omega; \R^2)\xrightarrow{}0.
\end{equation} 

Here $\mathbb S$ represents the space of ($2 \times 2$) symmetric matrices, and the space $H(\rot,\Omega;\mathbb S) := \{ \bm \sigma \in L^2(\Omega; \mathbb S) : \rot \bm \sigma \in L^2(\Omega; \mathbb R^2)\}$ contains symmetric matrix-valued functions $\bm \sigma$, such that $\bm \sigma$ and $\rot \bm \sigma$ are square integrable, where the $\rot$ operator acts rowwise on $\bm \sigma$.

\begin{remark}
In two dimensions, the gradgrad complex can be regarded as a rotated elasticity complex. The continuous elasticity complex reads as (cf. \cite{2002ArnoldWinther})
\begin{equation}\label{eq:2D:elasticity} 
     P_1 \xrightarrow{\subset} H^2(\Omega) \xrightarrow[]{\operatorname{Airy}} H(\div, \Omega; \bS) \xrightarrow[]{\operatorname{div}} L^2(\Omega; \R^2)\xrightarrow{}0,
\end{equation} 
whose finite element discretization has been studied in \cite{2002ArnoldWinther,2018ChristiansenHuHu}. Here the Airy function is defined as $\displaystyle  \begin{bmatrix}  \frac{\partial^2 }{\partial y^2}u &- \frac{\partial^2 u}{\partial x\partial y} \\ - \frac{\partial^2 u}{\partial x\partial y} & \frac{\partial^2 }{\partial x^2}u \end{bmatrix}$ for a scalar function $u \in H^2(\Omega)$.
\end{remark}

For a given integer $k \ge 3$, the finite element spaces $U_h$, $\bm \Sigma_h$ and $\bm Q_h$ are defined as follows. For convenience, the degrees of freedom are listed in a special way. Note that in what follows, the first sets \eqref{dof:1a}, \eqref{dof:2a} and \eqref{dof:3a} of degrees of freedom play an important role in the analysis of this paper.

\subsubsection*{$H^2$ conforming finite element space} For the $H^2$ conforming finite element space $U_h$, the shape function space is taken as $P_{k+2}(\bm f)$. For $u \in P_{k+2}(\bm f)$, the degrees of freedom are defined as follows:
\begin{enumerate}[label=(1\alph*).,ref=1\alph*]
\item\label{dof:1a} the function value and first order derivatives $u(\bm x)$, $\frac{\partial}{\partial x}u(\bm x)$ and $\frac{\partial}{\partial y}u(\bm x)$ at each vertex $\bm x$ of $\bm f$;
\item the second order derivatives $\nabla^2 u (\bm x)$ at each vertex $\bm x$ of $\bm f$;
\item the moments of the second order tangential derivate $(\frac{\partial^2}{\partial \bm t^2} u , \frac{\partial^2}{\partial\bm t^2 } p)_{\bm e}$, for $p \in B_{\bm e, k-4}^2 := (\lambda_0\lambda_1)^3P_{k-4}(\bm e)$ on each edge $\bm e$ of $\bm f$; 
\item the moments of the tangential-normal derivative $(\frac{\partial^2}{\partial \bm n \partial \bm t} u, \frac{\partial}{\partial \bm t} p)_{\bm e}$, for $p \in B_{\bm e, k-3}^1 := (\lambda_0\lambda_1)^2P_{k-3}(\bm e)$ on each edge $\bm e$ of $\bm f$; 
\item the moments of the Hessian $(\hess u, \hess p)_{\bm f}$, where $p \in B^1_{\bm f, k-4} := (\lambda_0\lambda_1\lambda_2)^2P_{k-4}(\bm f)$ inside $\bm f$.
\end{enumerate}

The above degrees of freedom are unisolvent, and the resulting finite element space is the Argyris finite element space \cite{1968Argyris}:
$$U_h = \{u \in C^1(\Omega) : u|_{\bm f} \in P_{k+2}(\bm f), \forall \bm f \in \mathcal T;  u \text{ is } C^2 \text{ at each vertex of } \mathcal T\}.$$

\subsubsection*{$H(\rot; \mathbb S)$ conforming finite element space}
For the $H(\rot; \mathbb S)$ conforming space $\bm \Sigma_h$, the shape function space is taken as $P_{k}(\bm f; \mathbb S)$, containing the symmetric matrix-valued polynomials of degree not greater than $k$. For $\bm \sigma \in P_{k}(\bm f; \mathbb S)$, the degrees of freedom are as follows:
\begin{enumerate}[label=(2\alph*).,ref=2\alph*]
\item \label{dof:2a} the moments of the tangential component $(\bm \sigma\bm t, \bm w)_{\bm e}$, for $\bm w \in RT(\bm f)$;
\item \label{dof:2b} the function value $\bm \sigma(\bm x)$ at each vertex $\bm x$ of $\bm f$; 
\item \label{dof:2c} the moments of the tangential-tangential component,$(\bm t^T \bm \sigma\bm t, p)_{\bm e}$, for $p \in B_{\bm e, k-2}/P_1(\bm e) := (\lambda_0 \lambda_1) P_{k-2}(\bm e)/ P_1(\bm e)$ on each edge $\bm e$ of $\bm f$;   
\item \label{dof:2d} the moments of the tangential-normal component $(\bm n^T \bm \sigma \bm t, p)_{\bm e}$, for $p \in  B_{\bm e, k-2}/P_0(\bm e)$ on each edge $\bm e$ of $\bm f$; 
\item \label{dof:2e} the moments inside $\bm f$ under the following inner product
\begin{equation}\label{eq:harminn-HrotS}(\mathcal P_{\hess  B^1_{\bm f, k-4} } \bm \sigma, \mathcal P_{\hess  B^1_{\bm f, k-4} } \bm \eta)_{\bm f} + (\rot \bm \sigma, \rot \bm \eta)_{\bm f} \text{ for }\bm \eta \in B_{\bm f,k}^{\rot;\bS},\end{equation}
 where 
$$B_{\bm f,k}^{\rot;\bS} := \{ \bm \sigma \in P_{k}(\bm f;\mathbb S) : \bm \sigma\bm t = 0 \text{ on } \partial \bm f \},$$
and $\mathcal P_{\hess  B^1_{\bm f, k-4}} $ is the $L^2$ orthogonal projection to the space $\hess  B^1_{\bm f, k-4}$. 

\end{enumerate} 

To show \eqref{eq:harminn-HrotS} is an inner product on $B_{\bm f,k}^{\rot;\bS}$, it suffices to prove that for $\bm \eta \in B_{\bm f,k}^{\rot;\bS}$ with $\rot \bm \eta = 0$, then there exists a function $u \in B_{\bm f, k-4}^1$ such that $\bm \eta = \hess u$. This will be proved in \Cref{lem:exact-2d-facebubble}.

\begin{remark}
\label{rmk:2d-rt}
Note that the space $RT$ restricted on edge $\bm e$ of $\bm f$ is the space $\{ a_n\bm n + a_t \bm t :a_{n} \in P_0(\bm e), a_t \in P_1(\bm e)\}$, the degrees of freedom \eqref{dof:2a}, \eqref{dof:2c} and \eqref{dof:2d} are indeed the moments $(\bm \sigma\bm t, \bm p)_{\bm e}$ for $\bm p \in [(\lambda_0\lambda_1)P_{k-2}(\bm e)]^2$. 
\end{remark}

The above degrees of freedom are unisolvent, and the resulting finite element space is a rotation of the Hu--Zhang $H(\div;\mathbb S)$ conforming finite element space  \cite{2015HuZhang}, 
$$
\Sigma_h := \{ \bm \sigma \in H(\rot;\bS) : \bm \sigma \in P_{k}(\bm f; \bS), \,\forall \bm f \in \mathcal T; \bm \sigma \text{ is } C^0 \text{ at each vertex of } \mathcal T\}.$$
 In particular,  the face bubble $B_{\bm f,k}^{\rot;\bS}$ is characterized as
\begin{equation}\label{eq:BrotS}
B_{\bm f,k}^{\rot;\bS}=  \{ p_0 \lambda_0 \lambda_{1} \bm n_{2}\bm n_{2}^T + p_1 \lambda_1 \lambda_2 \bm n_{0}\bm n_{0}^T+ p_2 \lambda_2\lambda_0\bm n_{1}\bm n_{1}^T : p_0,p_1,p_2 \in P_{k-2}(\bm f)\},
\end{equation}
where $\bm n_i$, $i = 0,1,2$, are the unit normal vector with respect to edge $\bm e_i$, and $\lambda_i$, $i = 0,1,2$, are the barycenter coordinates with respect to edge $\bm e_i$. See \cite{2015HuZhang} for more details.

    For completeness, the proof of the unisolvency is provided here.
\begin{proposition}
The above degrees of freedom, namely, \eqref{dof:2a}--\eqref{dof:2e} are unisolvent, provided that the bilinear form \eqref{eq:harminn-HrotS} is an inner product on $B_{\bm f, k}^{\rot; \mathbb S}$.
\end{proposition}
\begin{proof}
By \Cref{rmk:2d-rt} and \eqref{eq:BrotS}, the total number of the degrees of freedom is 
$$ 3 \times 3 + 3 \times 2(k-1) + 3 \times \frac{1}{2}(k-1)k = \frac{3}{2}(k+1)(k+2),$$ 
which is equal to the dimension of $P_k(\bm f; \mathbb S)$. It then suffices to show that for $\bm \sigma \in P_k(\bm f; \mathbb S)$,  if $\bm \sigma$ vanishes at all the degrees of freedom \eqref{dof:2a}-\eqref{dof:2e}, then $\bm \sigma = 0$.

Since $\bm \sigma$ vanishes at \eqref{dof:2b}, it holds that $\bm \sigma|_{\bm e} = (\lambda_0\lambda_1)P_{k-2}(\bm e)$. It then follows from \eqref{dof:2a}, \eqref{dof:2c} and \eqref{dof:2d} that $(\bm \sigma \bm t)_{\bm e} = 0$ for each edge $\bm e$ of $\bm f$, implying that $\bm \sigma \in B_{\bm f,k}^{\rot; \mathbb S}$. Finally, the fact that \eqref{eq:harminn-HrotS} is an inner product on the space $B_{\bm f,k}^{\rot; \mathbb S}$ completes the proof.
\end{proof}

\subsubsection*{$L^2(\R^2)$ finite element space}
For the $L^2(\R^2)$ finite element space $\bm Q_h$, the shape function space is taken as $P_{k-1}(\bm f;\mathbb R^2) $. For $\bm q \in P_{k-1}(\bm f;\mathbb R^2)$, the local degrees of freedom are as follows:
\begin{enumerate}[label=(3\alph*).,ref=3\alph*]
\item\label{dof:3a} the moments $(\bm q, \bm w)_{\bm f}$ for $\bm w \in RT(\bm f);$
\item the moments $(\bm q, \bm w)_{\bm f}$ for $\bm w \in P_{k-1}(\bm f; \mathbb R^2) /{RT(\bm f)}$.
\end{enumerate}

The lowest order ($k = 3$) of the complex \eqref{eq:2D:hess}, which starts from the quintic Argyris element, is illustrated as follows, where $\dim U_h = 6|\mathsf V| + |\mathsf E|$, $\dim \bm \Sigma_h = 3|\mathsf V| + 4|\mathsf E| + 9|\mathsf F|$ and $\dim \bm V_h = 12|\mathsf F|.$

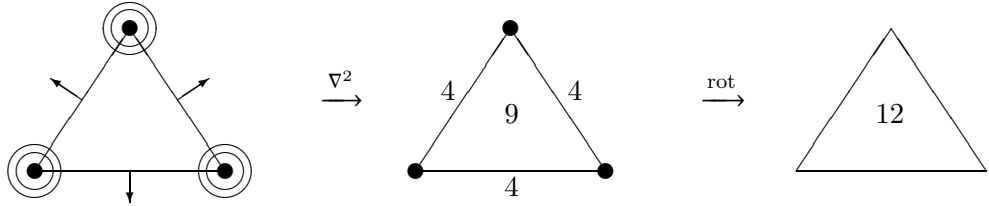
\begin{figure}[H] \begin{center} \setlength{\unitlength}{1.20pt}

    \begin{picture}(270,45)(0,5)
\put(0,0){\begin{picture}(70,70)(0,0)\put(0,0){\line(1,0){60}}
    \put(0,0){\line(2,3){30}}\put(60,0){\line(-2,3){30}}
    \put(0,0){\circle*{5}}\put(0,0){\circle{12}}\put(0,0){\circle{18}}
    \put(60,0){\circle*{5}}\put(60,0){\circle{12}}\put(60,0){\circle{18}}
    \put(30,45){\circle*{5}}\put(30,45){\circle{12}}\put(30,45){\circle{18}}
   \put(15,22.5){\vector(-3,2){10}}\put(45,22.5){\vector(3,2){10}}
    \put(30,0){\vector(0,-1){10}}
   \end{picture}}
  
\put(90,20){$\xrightarrow[]{\hess}$}
\put(120,0){\begin{picture}(70,70)(0,0)\put(0,0){\line(1,0){60}}
    \put(0,0){\line(2,3){30}}\put(60,0){\line(-2,3){30}}
    \put(0,0){\circle*{5}}
    \put(60,0){\circle*{5}}
    \put(30,45){\circle*{5}}
    \put(8,22.5){$4$}
    \put(48,22.5){$4$}
    \put(28,-8){$4$}
    \put(28,15){$9$}
   \end{picture}}
   \put(210,20){$\xrightarrow[]{\rot}$}
  \put(240,0){\begin{picture}(70,70)(0,0)\put(0,0){\line(1,0){60}}
    \put(0,0){\line(2,3){30}}\put(60,0){\line(-2,3){30}}
    \put(25,15){$12$}
   \end{picture}}
\end{picture}
\end{center}
\end{figure}

\vspace{1em}
The main result in two dimension is shown as follows, and the proof is in \Cref{sec:proof-main-2d}.
\begin{theorem}
	\label{thm:main-2d}
	There exist operators $\pi^l$, $l = 0,1,2$, such that $\pi^0 : H^2(\Omega) \to U_h$, $\pi^1 : H(\rot,\Omega; \mathbb S) \to \Sigma_h$ and $\pi^2 : L^2(\Omega;\mathbb R^2) \to \bm Q_h$ are all projection operators, and the following diagram commutes.
	\begin{equation}
		\label{eq:maincd-2d}
		\xymatrix{
		P_1  \ar[r] & H^2(\Omega) \ar[d]_-{\pi^0}\ar[r]^-{\hess} & H(\rot,\Omega; \mathbb S) \ar[d]_-{\pi^1} \ar[r]^-{\rot} & L^2(\Omega;\mathbb R^2) \ar[d]_-{\pi^2} \ar[r] & 0 \\
		P_1 \ar[r] & U_h \ar[r]^-{\hess} &\bm \Sigma_h \ar[r]^-{\rot} & \bm Q_h \ar[r] & 0
		}
	\end{equation}
	Namely, $\hess \pi^0 = \pi^1 \hess$, $\pi^2 \curl = \curl \pi^1$, and these operators $\pi^l$, $l = 0,1,2,$ are locally determined. For example, the value of $\pi^0 u $ on $\bm f$ is determined by the value of $u$ on $\omega_{\bm f}^{[1]}$.

	If moreover, $\mathcal T$ is shape-regular, then the projection operators are locally bounded, i.e.,
	\begin{equation}   \label{eq:thmbound-2d}\|\pi^0 u\|_{H^2(\bm f)} \le C\|u\|_{H^2(\omega_{\bm f}^{[1]})}, \|\pi^1 \bm \sigma\|_{H(\rot, f; \mathbb S)} \le C\|\bm \sigma\|_{H(\rot, \omega^{[2]}_{\bm f};\mathbb S)}, \|\pi^2 \bm v\|_{L^2(\bm f)} \le C\|\bm v\|_{L^2(\omega^{[2]}_{\bm f})}.\end{equation}
	Here the constant $C$ only depends on the shape regularity constant and the polynomial degree. As a consequence, all the operators are globally bounded, i.e., $\pi^0$ is $H^2$ bounded, $\pi^1$ is $H(\rot;\mathbb S)$ bounded and $\pi^2$ is $L^2$ bounded.

\end{theorem}

\begin{remark} The above theorem also implies the existence of the local bounded commuting projection operators $\pi^l, l=0,1,2,$ from the continuous elasticity complex in two dimensions to the finite element one,
	\begin{equation}
		\label{eq:cd-elasticity}
		\xymatrix{
		P_1  \ar[r] & H^2(\Omega) \ar[d]_-{\pi^0}\ar[r]^-{\operatorname{Airy}} & H(\div,\Omega; \mathbb S) \ar[d]_-{\pi^1} \ar[r]^-{\div} & L^2(\Omega;\mathbb R^2) \ar[d]_-{\pi^2} \ar[r] & 0 \\
		P_1 \ar[r] & \mathbf{Ar}_{k+2} \ar[r]^-{\operatorname{Airy}} & \mathbf{HZ}_{k} \ar[r]^-{\div} & \mathbf{P}_{k-1}^{-} \ar[r] & 0,
		}
	\end{equation}
where $\mathbf{Ar}_{k+2}$ is the Argyris finite element space \cite{1968Argyris}, $\mathbf{HZ}_{k}$ is the Hu--Zhang $H(\div; \mathbb S)$ finite element space \cite{2015HuZhang}, and $\mathbf{P}_{k-1}^{-}$ is the DG space.
\end{remark}

\subsection{Finite element gradgrad complex in three dimensions}
\label{sec:dof-3d}
For three dimensions, suppose that a contractible polyhedral domain $\Omega$ and a triangulation $\mathcal T$ are given. Denote by $\mathsf V$ the set of vertices, $\mathsf E$ the set of edges, $\mathsf F$ the set of faces, $\mathsf K$ the set of cells, and $\mathsf S = \mathsf V\cup \mathsf E \cup \mathsf F \cup \mathsf K$ the set of all simplices. The notations $\bm x$, $\bm e$, $\bm f$, $\bm K$ are used to represent any vertex, edge, face and element of the mesh $\mathcal T$, respectively.

With slight abuse of notations, let $\bx = [x_1, x_2, x_3]^T$, $\be_1 = [1,0,0]^T$, $\be_2 = [0,1,0]^T$ and $\be_3 = [0,0,1]^T$ be vectors in three dimensions. Let $\omega_{\sigma}^{[m]}$ and $\omega_{\sigma}^h$ be defined similarly as those in two dimensions.
Let $P_k(\bm K)$ be the space of polynomials of degree not greater than $k$, let $RT = RT(\bm K) := \{ \bm a + b \bx: \bm a \in \mathbb R^3, b \in \mathbb R\}$ be the lowest order Raviart--Thomas space in three dimensions. Note that in the subsequent construction, the two-dimensional Raviart--Thomas space $RT(\bm f)$ will also appear.

 For a face $\bm f$ of the element $\bm K$, let $\bm n$ be its unit normal vector, $\bm t_i, i = 1,2,$ be its two linearly independent unit tangential vectors, such that $\bm t_1 \perp \bm t_2$. For each edge of the element, let $\bm t$ be its unit tangential vector, $\bm n_i, i = 1,2$, be its two linearly independent unit normal vectors, such that $\bm n_1 \perp \bm n_2$. Define the operator $E_{\bm f} : \mathbb R^3 \to \mathbb R^2$ such that $E_{\bm f}\bm v = [\bm v \cdot \bm t_1, \bm v\cdot \bm t_2]^T$ for $\bm v \in \mathbb R^3$. Define the operator $\mathbb E_{\bm f}$ the tangential-tangential component of $\bm \sigma$ such that $\displaystyle \mathbb E_{\bm f}\bm \sigma = \begin{bmatrix} \bm t_1^T \bm \sigma \bm t_1 & \bm t_1^T \bm \sigma \bm t_2 \\ \bm t_2^T \bm \sigma \bm t_1 & \bm t_2^T \bm \sigma \bm t_2 \end{bmatrix}.$ Let $\nabla_{\bm f}$ be the gradient operator with respect to $\bm t_1$ and $\bm t_2$, such that $\nabla_{\bm f} u = [\bm t_1 \cdot \nabla u, \bm t_2 \cdot \nabla u]^T,$ and $\rot_{\bm f} \bm v = - \frac{\partial v_2}{\partial \bm t_1} + \frac{\partial v_1}{\partial \bm t_2}$ are defined for a scalar function $u$ and an $\mathbb R^2$ valued-function $\bm v$, respectively.

Recall the finite element gradgrad complex in three dimensions \cite{2021HuLiang},
\begin{equation}\label{eq:3D:hess} 
    P_1 \xrightarrow{\subset} U_h \xrightarrow[]{\operatorname{\hess}} \bm \Sigma_h \xrightarrow[]{\operatorname{curl}} \bm \Xi_h  \xrightarrow[]{\operatorname{div}} \bm Q_h\xrightarrow{}0,
\end{equation} 
which is a discretization of the following continuous Hessian complex,
\begin{equation}\label{eq:3D:hesscon} 
     P_1 \xrightarrow{\subset} H^2(\Omega) \xrightarrow[]{\operatorname{\hess}} H(\curl, \Omega; \bS) \xrightarrow[]{\operatorname{curl}} H(\div, \Omega;\mathbb T) \xrightarrow[]{\operatorname{div}} L^2(\mathbb R^3)\xrightarrow{}0,
\end{equation} 
Here $\mathbb S$ and $\mathbb T$ represent the spaces of $(3 \times 3)$ symmetric matrices and traceless matrices, respectively, the space
$$H(\curl, \Omega; \bS) := \{ \bm \sigma \in L^2(\Omega; \mathbb S) :\curl \bm \sigma \in L^2(\Omega; \mathbb T)\}$$
contains symmetric matrix-valued functions $\bm \sigma$ such that $\bm \sigma$ and $\curl \bm \sigma$ are square integrable, and the space 
$$H(\div, \Omega; \bT) := \{ \bm v \in L^2(\Omega; \mathbb T) : \div \bm v \in L^2(\Omega; \mathbb R^3)\}$$
contains traceless matrix-valued functions $\bm v$ such that $\bm v$ and $\div \bm v$ are square integrable. Note that the operators $\curl$ and $\div$ act rowwise. 
Here the degrees of freedom of the finite element spaces $U_h, \bm \Sigma_h, \bm V_h$ and $\bm Q_h$ are proposed for a given integer $k \ge 7$, with a slight modification from those of \cite{2021HuLiang}. However, the finite element spaces are the same, since only the inner products on each sub-simplex are replaced by the corresponding harmonic inner products, cf. \cite{2023HuLiangLin}. 

\subsubsection*{$H^2$ conforming finite element space}
For the $H^2$ conforming element space $U_h$, the shape function space is defined as $P_{k+2}(\bm K)$. For $u \in P_{k+2}(\bm K)$, the local degrees of freedom are as follows:
\begin{enumerate}[label=(4\alph*).,ref=4\alph*]
\item\label{dof:4a} the function value and first order derivatives $u(\bm x)$, $\nabla u(\bm x)$ at each vertex $\bm x$;
\item\label{dof:4b} the second to fourth order derivatives $D^{\alpha} u (\bm x)$ at each vertex $\bm x$, for $2 \le |\alpha| \le 4$;
\item\label{dof:4c} the moments of the second order tangential derivative $(\frac{\partial^2}{\partial \bm t^2} u , \frac{\partial^2}{\partial \bm t^2} p)_{\bm e}$, for $p \in (\lambda_0\lambda_1)^5P_{k-8}(\bm e)$ on each edge $\bm e$ of $\bm K$;
\item\label{dof:4d} the moments of the tangential-normal derivatives $(\frac{\partial^2}{\partial \bm n_{j} \partial \bm t} u, \frac{\partial}{\partial \bm t} p)_{\bm e}$, for $p \in  (\lambda_0\lambda_1)^4P_{k-7}(\bm e)$ and $j = 1,2$ on each edge $\bm e$ of $\bm K$;
\item\label{dof:4e} the moments of the second order normal derivatives
$(\frac{\partial^2}{\partial \bm n_{j} \partial \bm n_{j'}} u, p)_{\bm e}$ for $p \in  (\lambda_0\lambda_1)^3P_{k-6}(\bm e)$ and $j, j' =1,2$ on each edge $\bm e$ of $\bm K$;
\item\label{dof:4f} the moments of the face Hessian $(\nabla_{\bm f}^2 u, \nabla_{\bm f}^2 p)_{\bm f}$, for $p\in B_{\bm f,k-7}^2 := (\lambda_0\lambda_1\lambda_2)^3P_{k-7}(\bm f)$, on each face $\bm f$ of $\bm K$;
\item\label{dof:4g} the moments of the face gradient of $\frac{\partial}{\partial\bm n}u$, namely, $(\nabla_{\bm f} (\frac{\partial}{\partial\bm n}u), \nabla_{\bm f}p)_{\bm f}$, for $p\in B_{\bm f, k-5}^1 :=  (\lambda_0\lambda_1\lambda_2)^2P_{k-5}(\bm f),$ on each face $\bm f$ of $\bm K$;
\item\label{dof:4h} the moments of the Hessian insides $\bm K$, namely, $(\hess u,\hess p)_{\bm K}$, for $p \in B_{\bm K, k-6}^1:=(\lambda_0\lambda_1\lambda_2\lambda_3)^2P_{k-6}(\bm K)$.
\end{enumerate}

The above degrees of freedom are unisolvent, and the resulting finite element space \cite{2009Zhang} is
$$U_h := \{ u \in C^1(\Omega); u|_{\bm K} \in P_{k+2}(\bm K), \forall \bm K \in \mathcal T; u \text{ is } C^4 \text{ at each vertex of } \mathcal T, C^2 \text{ on each edge of } \mathcal T \}.$$
\subsubsection*{$H(\operatorname{curl};\mathbb{S})$ conforming finite element space}
For the $H(\operatorname{curl};\mathbb{S})$ conforming finite element space $\bm \Sigma_h$, the shape function space is defined as $P_{k}(\bm K;\mathbb{S})$. For $\bm{\sigma} \in P_{k}(\bm K;\mathbb{S})$, the local degrees of freedom are as follows:
\begin{enumerate}[label=(5\alph*).,ref=5\alph*]
	\item \label{dof:5a}the moments $(\bm \sigma\bm t, \bm w)_{\bm e}$ of the tangential component, for $\bm w \in RT(\bm K)$, on each edge $\bm e$ of $\bm K$;
	\item\label{dof:5b} the function values and first and second order derivatives at each vertex $D^{\alpha}\bm{\sigma}(\bm{x})$, $0\leq |\alpha| \leq 2$;
	\item \label{dof:5c} the moments of the tangential-tangential component $(\bm t^{T}\bm{\sigma}\bm t,p)_{\bm e}$, for $p \in (\lambda_0 \lambda_1)^3 P_{k-6}(\bm e)/ P_1(\bm e)$, on each edge $\bm e$ of $\bm K$;
	\item \label{dof:5d} the moments of the normal-tangential components $(\bm n_{j}^{T}\bm{\sigma}\bm t,p)_{\bm e}$, for $p\in (\lambda_0\lambda_1)^3P_{k-6}(\bm e)/P_0(\bm e)$ and $j = 1,2$ on each edge $\bm e$ of $\bm K$;
	\item \label{dof:5e} the moments of the normal-normal components $(\bm n_{j}^{T}\bm{\sigma}\bm n_{j'},p)_{\bm e}$, for $p\in (\lambda_0\lambda_1)^3P_{k-6}(\bm e)$ and $j, j' = 1,2$, on each edge $\bm e$ of $\bm K$;
	\item \label{dof:5f} the moments of the face tangential-tangential components under the inner product
    \begin{equation}\label{eq:inner-5f}(\mathcal{P}_{\hess_{\bm f} B_{\bm f, k-7}^2}\mathbb{E}_{\bm f}\bm{\sigma},\mathcal{P}_{\hess_{\bm f} B_{\bm f, k-7}^2} \bm \eta)_{\bm f} + (\operatorname{rot}_{\bm f}\mathbb{E}_{\bm f}\bm{\sigma},\operatorname{rot}_{\bm f}\bm \eta)_{\bm f}
    \end{equation} for $\bm \eta\in (\lambda_0\lambda_1\lambda_2) P^{(0)}_{k-3}(\bm f; \bS_{2\times 2})$ on each face $\bm f$ of $\bm K$, where $$P_{k-3}^{(0)}(\bm f) := \{ \lambda_0\lambda_1\lambda_2 p: p \in P_{k-3}(\bm f), p \text{ vanishes at all vertices of }\bm f\};$$ 
	\item \label{dof:5g} the moments of the face normal-tangential components under the inner product \begin{equation}\label{eq:inner-5g}(\mathcal{P}_{\operatorname{grad}_{\bm f} B_{\bm f, k-5}^1}{E}_{\bm f}(\bm{\sigma}\bm n),\mathcal{P}_{\operatorname{grad}_{\bm f} B_{\bm f, k-5}^1} \bm w)_{\bm f} + (\operatorname{rot}_{\bm f}{E}_{\bm f}(\bm{\sigma}\bm n),\operatorname{rot}_{\bm f}\bm w)_{\bm f}\end{equation} for $\bm w\in (\lambda_0\lambda_1\lambda_2) P^{(0)}_{k-3}(\bm f;\mathbb R^2),$ on each face $\bm f$ of $\bm K$;
	\item \label{dof:5h} the moments inside $\bm K$  \begin{equation}\label{eq:inner-5h}(\mathcal P_{\hess B_{\bm K, k-6}^1} \bm \sigma,\mathcal P_{\hess B_{\bm K, k-6}^1} \bm \eta)_{\bm K} + (\curl \bm \sigma, \curl \bm \eta)_{\bm K}\end{equation} for $\bm \eta\in B_{\bm K,k}^{\curl;\bS}$. Here \begin{equation}\label{eq:BcurlS} B_{\bm K,k}^{\curl;\bS}:= \{ \sum_{i=0}^3\lambda_{j}\lambda_{l}\lambda_{m}P_{k-3}^{(0)}(\bm K)\bm{n}_{i}\bm{n}_{i}^{T}\}  +  (\lambda_0\lambda_1\lambda_2\lambda_3)P_{k-4}(\bm K)\end{equation}
	with $\{i,j,l,m\}$ is a permutation of $\{0,1,2,3\}$, and 
    \begin{equation} \label{eq:P0} P_{k-3}^{(0)}(\bm K): =\{p\in P_{k-3}(\bm K): p \text{ vanishes at all vertices of }\bm K\}.\end{equation}
\end{enumerate}

\begin{remark}
To show that \eqref{eq:inner-5f} is an inner product on the space $(\lambda_0\lambda_1\lambda_2)P_{k-3}^{(0)}(\bm f;\mathbb S_{2\times 2})$, it suffices to check that if $\rot_{\bm f} \bm \sigma$ = 0 for some $\bm \sigma \in (\lambda_0\lambda_1\lambda_2)P_{k-3}^{(0)}(\bm f;\mathbb S_{2\times 2})$, then there exists $u \in B_{\bm f, k-7}^2$ such that $\bm \sigma = \hess_{\bm f} u $, which is proved in \Cref{lem:exact-facebubble-1-3d}. Similarly, \Cref{lem:exact-facebubble-2-3d} implies that \eqref{eq:inner-5g} is an inner product on the space $(\lambda_0\lambda_1\lambda_2) P^{(0)}_{k-3}(\bm f;\mathbb R^2)$, and \Cref{lem:exact-elebubble-3d} implies that \eqref{eq:inner-5h} is an inner product on the space $B_{\bm K,k}^{\curl;\bS}$.
\end{remark}
\begin{remark}
    The degrees of freedom \eqref{dof:5a}, \eqref{dof:5c} and \eqref{dof:5d} are indeed the moments $(\bm \sigma\bm t, \bm w)_{\bm e}$ for $\bm w \in [(\lambda_0\lambda_1)^3P_{k-6}(\bm e)]^2$. 
    \end{remark}

The above degrees of freedom are unisolvent, and the resulting finite element space~\cite{2021HuLiang} is
$$\bm \Sigma_h = \{ \bm \sigma : H(\curl,\Omega;\mathbb S) : \bm \sigma|_{\bm K} \in P_k(\bm K;\mathbb S), \forall \bm K \in \mathcal T; \bm \sigma \text{ is } C^2 \text{ at each vertex, } C^0 \text{ on each edge}\}.$$

    The unisolvency is proved in what follows.
\begin{proposition}
The above degrees of freedom, namely, \eqref{dof:5a}--\eqref{dof:5h} are unisolvent, provided that \eqref{eq:inner-5f}, \eqref{eq:inner-5g}, and \eqref{eq:inner-5h} are inner products on the space $(\lambda_0\lambda_1\lambda_2)P_{k-3}^{(0)}(\bm f;\mathbb S_{2\times 2})$, $(\lambda_0\lambda_1\lambda_2) P^{(0)}_{k-3}(\bm f;\mathbb R^2)$, and $B_{\bm K,k}^{\curl;\bS}$, respectively.
\end{proposition}
\begin{proof}
First, note that $$\dim (\lambda_0\lambda_1\lambda_2)P_{k-3}^{(0)}(\bm f;\mathbb S_{2\times 2}) = 3(\frac{1}{2}(k-2)(k-1) - 3) ,$$ and $$\dim (\lambda_0\lambda_1\lambda_2) P^{(0)}_{k-3}(\bm f;\mathbb R^2) = 2(\frac{1}{2}(k-2)(k-1) - 3).$$ Recall from \cite[Section 3.2]{2021HuLiang} that the dimension of $B_{\bm K,k}^{\curl;\bS}$ is $k^3-4k^2+5k-14$. Therefore, the total number of the degrees of freedom are 
\begin{equation*}\begin{split}
    & 4 \times 60 + 6 \times 6(k-5) + 4 \times 5(\frac{1}{2}(k-2)(k-1) - 3) + (k^3-4k^2+5k-14)  \\  = &(k+1)(k+2)(k+3),
\end{split}\end{equation*}
which is equal to the dimension of $P_{k}(\bm K; \mathbb S)$.

It suffices to show that for $\bm \sigma \in P_{k}(\bm K; \mathbb S)$, if $\bm \sigma$ vanishes at all the degrees of freedom, then $\bm \sigma = 0$. If follows from \eqref{dof:5a}-\eqref{dof:5e} such that $\bm \sigma = 0$ on each edge $\bm e$ of $\bm K$. On each face $\bm f$ of $\bm K$, the set of degrees of freedom \eqref{dof:5b} indicates that $$\mathbb E_{\bm f}\bm \sigma \in  (\lambda_0\lambda_1\lambda_2)P_{k-3}^{(0)}(\bm f;\mathbb S_{2\times 2}).$$ As a result, it follows from \eqref{dof:5e} that $\mathbb E_{\bm f}\bm \sigma = 0$. Similarly, it holds that $E_{\bm f}(\bm \sigma \bm n ) = 0$. Therefore, this yields $\bm \sigma \times \bm n = 0$, which implies that $\bm \sigma \in B_{\bm K,k}^{\curl;\bS}$, according to \cite[Theorem 3.2]{2021HuLiang}. Finally, the last set of degrees of freedom \eqref{dof:5h} shows that $\bm \sigma = 0$, which completes the proof.
\end{proof}

\subsubsection*{$H(\operatorname{div};\mathbb{T})$ conforming finite element space}
For the $H(\operatorname{div};\mathbb{T})$ conforming finite element space $\bm V_h$, the shape function space is taken as $P_{k-1}(\bm K;\mathbb{T})$. For $\bm{v} \in P_{k-1}(\bm K;\mathbb{T})$, the local degrees of freedom are as follows:
\begin{enumerate}[label=(6\alph*).,ref=6\alph*]
	\item \label{dof:6a}the moments of its normal component $(\bm v\bm n, \bm w)_{\bm f}$, for $\bm w \in RT(\bm K)$, on each face $\bm f$ of $\bm K$;
	\item \label{dof:6b} the function values and first order derivatives at each vertex: $D^{\alpha}\bm{v}(\bm{x}),0\leq |\alpha| \leq 1$;
	\item \label{dof:6c} the moments of its normal-tangential component $(E_{\bm f}(\bm v \bm n), \bm w)_{\bm f}$, for $\bm w \in P_{k-1}^{(1)}(\bm f; \mathbb R^2) /RT(\bm f)$ on each face $\bm f$ of $\bm K$, where 
    \begin{equation}\label{eq:P1} P_{k-1}^{(1)}(\bm f) := \{p \in P_{k-1}(\bm f): p \text{ and } \nabla p \text{ vanishes at vertices}\};\end{equation}
	\item \label{dof:6d} the moments of its normal-normal component $(\bm n^T \bm v \bm n, p)_{\bm f}$, for $p\in P_{k-1}^{(1)}(\bm f)/\mathbb R$, on each face $\bm f$ of $\bm K$;
    \item \label{dof:6e} the moments inside $\bm K$, 
    \begin{equation}\label{inner-6e} (\mathcal P_{\curl B_{\bm K,k}^{\curl;\bS}}\bm v, \mathcal P_{\curl B_{\bm K,k}^{\curl;\bS}} \bm \eta)_{\bm K} + (\div \bm v, \div \bm \eta)_{\bm K}\end{equation} for $\bm \eta\in B_{\bm K,k-1}^{\div;\mathbb T}$. Here
    \begin{equation} \label{eq:BdivT} B_{\bm K,k-1}^{\div;\mathbb T} : = \sum_{i=0}^3\sum_{\substack{0\leq j < l\leq 3 \\ j,l\neq i}}\lambda_{j}\lambda_{l}P_{k-3}^{j,l,0}(\bm K)\bm{n}_{i}\bm{t}_{j,l}^{T}\end{equation}
    with $P_{k-3}^{j,l,0}(\bm K):= \{u\in P_{k-3}(\bm K): u \text{ vanishes at vertices } \bm{x}_{j},\bm{x}_{l}\}$, and $\bm t_{j,l}$ is the tangential vector of the edge connecting $\bm x_j$ and $\bm x_l$.
\end{enumerate}

The above degrees of freedom are unisolvent, and the resulting finite element space is 
$$\bm V_h = \{ \bm v \in H(\div,\Omega;\mathbb T) : \bm v|_{\bm K} \in P_{k-1}(\bm K;\mathbb T), \forall \bm K \in \mathcal T; \bm v \text{ is } C^1 \text{ at each vertex of }\mathcal T\}.$$ 
The unisolvency of the $H(\div;\mathbb T)$ conforming finite element $ \bm V_h$ can be proved similarly.
\subsubsection*{$L^2(\mathbb R^3)$ finite element space} For the $L^2$ finite element space $\bm Q_h$, the shape function is taken as $P_{k-2}(K;\mathbb R^3)$. For $\bm v \in P_{k-2}(K;\mathbb R^3)$, the local degrees of freedom are as follows:
\begin{enumerate}[label=(7\alph*).,ref=7\alph*]
\item \label{dof:7a} the moments in $\bm K$, namely, $(\bm q, \bm w)_{\bm K}$ for $\bm w \in RT(\bm K)$;
\item \label{dof:7b} the function value $\bm q(\bm x)$ at each vertex $\bm x$ of $\bm K$; 
\item \label{dof:7c} the moments inside $\bm K$, namely, $(\bm q,  \bm w)_{\bm K}$ for $ \bm w\in P_{k-2}^{(0)}(\bm K;\mathbb R^3) / RT(\bm K)$.
\end{enumerate}

The above degrees of freedom are unisolvent, and the resulting finite element space is
$$\bm Q_h = \{ \bm q \in L^2(\Omega;\mathbb R^3) : \bm q|_{\bm K} \in P_{k-2}(\bm K;\mathbb R^3), \forall \bm K \in \mathcal T; \bm q \text{ is } C^0 \text{ at each vertex of } \mathcal T\}.$$ 
The main result in three dimensions is stated as follows, and the proof is in \Cref{sec:3d}.

\begin{theorem}
    \label{thm:main-3d}
    There exist operators $\pi^l$, $l = 0,1,2,3$, such that $\pi^0 : H^2(\Omega) \to U_h$, $\pi^1 : H(\curl,\Omega; \mathbb S) \to \bm \Sigma_h$, $\pi^2 :H(\div, \Omega; \mathbb T) \to \bm V_h$ and $\pi^3 : L^2(\Omega;\mathbb R^3) \to \bm Q_h$ are all projection operators,  and the following diagram commutes.
\begin{equation}
    \label{eq:cd3d}
        \xymatrix{
          P_1  \ar[r] & H^2(\Omega) \ar[d]_-{\pi^0}\ar[r]^-{\hess} & H(\curl, \Omega; \mathbb S) \ar[d]_-{\pi^1} \ar[r]^-{\curl} & H(\div; \Omega; \mathbb T) \ar[r]^-{\div}\ar[d]_-{\pi^2} &L^2(\Omega;\mathbb R^3) \ar[d]_-{\pi^3} \ar[r] & 0 \\
          P_1 \ar[r] & U_h \ar[r]^-{\hess} & \bm \Sigma_h \ar[r]^-{\curl} & \bm V_h \ar[r]^-{\div} &\bm Q_h \ar[r] & 0
        }
\end{equation}
Namely, $\hess \pi^0 = \pi^1 \hess$, $\pi^2 \curl = \curl \pi^1$, $\div \pi^2 = \pi^3 \div$, and these operators $\pi^l$, $l = 0,1,2,3,$ are locally determined. For example, the value of $\pi^0 u $ on $\bm K$ is determined by the value of $u$ on $\omega_{K}^{[1]}$.

	If moreover, $\mathcal T$ is shape-regular, then the projection operator is locally bounded, i.e.,
	\begin{equation}
        \label{eq:thmbound-3d}
        \|\pi^0 u\|_{H^2(\bm K)} \le C\|u\|_{H^2(\omega_{\bm K}^{[1]})}, \|\pi^1 \bm \sigma\|_{H(\curl, \bm K; \bS)} \le C\|\bm \sigma\|_{H(\curl,\omega^{[2]}_{\bm K}; \bS)},
    \end{equation}
    \begin{equation*} \|\pi^2 \bm v\|_{H(\div,\bm K; \mathbb T)} \le C\|\bm v\|_{H(\div, \omega_{\bm K}^{[3]}; \mathbb T)}, \| \pi^3 \bm p\|_{L^2(\bm K; \mathbb R^3)} \le C \| \bm p\|_{L^2(\omega_{\bm K}^{[3]};\mathbb R^3)}.
    \end{equation*}
	Here the constant $C$ only depends on the shape regularity and the reference finite element. As a consequence, all the operators are globally bounded, i.e., $\pi^0$ is $H^2$ bounded, $\pi^1$ is $H(\curl;\mathbb S)$ bounded, $\pi^2$ is $H(\div;\mathbb T)$ bounded, and  $\pi^3$ is $L^2(\mathbb R^3)$ bounded.

\end{theorem}

\section{Proof of \Cref{thm:main-3d}}
\label{sec:3d}
This section is devoted to the proof of \Cref{thm:main-3d}, which will be accomplished in several steps. The two-dimensional results are similar and will be briefly discussed in \Cref{sec:2d}. We provide a sketch of the proof in \Cref{sec:3d-sketch}, and the following sections give the detailed proof.
\subsection{Sketch of the proof}
\label{sec:3d-sketch}
Suppose that the basis functions with respect to the degrees of freedom \eqref{dof:4a}, \eqref{dof:5a}, \eqref{dof:6a} and \eqref{dof:7a} are 
$$\tilde{\varphi}_{\bm x, j}, \forall \bm x \in \mathsf V, \,\, \varphi_{\bm e, j}, \forall \bm e \in \mathsf E, \,\, \varphi_{\bm f, j}, \forall \bm f \in \mathsf F, \,\,\varphi_{\bm K, j}, \forall \bm K \in \mathsf K,$$ 
respectively. 
We first show that the following sequence is a complex. 
\begin{proposition}\label{prop:3D:hess-sk}
    The discrete sequence 
    \begin{equation}\label{eq:3D:hess-sk} 
    \begin{split}
    P_1 \xrightarrow[]{\subset} \Span & \{\tilde \varphi_{\bm x, j'} , \bm x \in \mathsf V, j' = 0,1,2,3\} \xrightarrow[]{\hess}   \Span\{\varphi_{\bm e, j'} , \bm e \in \mathsf E,  j' = 0,1,2,3\} \xrightarrow[]{\operatorname{curl}} \\ & \Span\{\varphi_{\bm f, j'} , \bm f \in \mathsf F, j' = 0,1,2,3\} \xrightarrow[]{\operatorname{div}} \Span\{\varphi_{\bm K, j'} , \bm K \in \mathsf K, j' = 0,1,2,3\} \xrightarrow[]{} 0
    \end{split}
    \end{equation}
    is a complex.
\end{proposition}
This will be proved in \Cref{sec:proof:hess-sk} below. 
    For convenience, assume that $\varphi_{\bm e,j}$ is the basis function with respect to the degrees of freedom $(\bm \sigma \bx, \bm t)_{\bm e}$ for $j = 0$, and $(\bm \sigma \be_j, \bm t)_{\bm e}$ for $j =1,2,3$. The basis functions $\varphi_{\bm f, j}$ and $\varphi_{\bm K,j}$ are defined similarly. By integration by parts, for $\bm \alpha \in \mathbb R^3$, $\beta \in \mathbb R$, $\bm a, \bm b \in \mathsf V$ and $\bm e = [\bm a, \bm b] \in \mathsf E$, it holds that
\begin{equation}
    \label{eq:int-by-part-hessu}
    \begin{split}
    \int_{\bm e} \bm t^T\nabla^2 u  (\bm \alpha + \beta \bx)
    = & 
    \int_{\bm e} \frac{\partial(\nabla u)}{\partial \bm t}\cdot (\bm \alpha + \beta \bx)  \\
    = &  (\nabla u\cdot(\bm \alpha + \beta \bx))|_{\bm a}^{\bm b} - \int_{\bm e} \nabla u \cdot \frac{\partial(\bm \alpha + \beta\bx)}{\partial \bm t} \\
    = & (\nabla u\cdot(\bm a + b \bx))|_{\bm a}^{\bm b} - \int_{\bm e} \frac{\partial u}{\partial \bm t}\frac{\partial(\bm \alpha + \beta\bx)\cdot \bm t }{\partial \bm t} - \sum_{i=1}^2\int_{\bm e} \frac{\partial u}{\partial \bm n_{i}}\frac{\partial(\bm \alpha + \beta \bx)\cdot \bm n_{i} }{\partial \bm t}\\
    = & (\nabla u\cdot(\bm \alpha + \beta \bx))|_{\bm a}^{\bm b} - \int_{\bm e} \frac{\partial u}{\partial \bm t}\beta \\
    = & (\nabla u\cdot(\bm \alpha + \beta \bx))(\bm b) - (\nabla u\cdot(\bm \alpha + \beta \bx))(\bm a) + \beta u(\bm a) - \beta u(\bm b).
    \end{split}
\end{equation}
This motivates to define the degrees of freedom: 
$\bx \cdot \nabla u(\bm x)- u(\bm x)$, $\be_1 \cdot \nabla u(\bm x) $, $\be_2 \cdot \nabla u(\bm x)$ and $\be_3 \cdot \nabla u(\bm x)$ for the space $\Span\{\tilde \varphi_{\bm x, j} : j = 0,1,2,3\}$. The corresponding basis functions are then denoted as $\varphi_{\bm x,0}, \varphi_{\bm x, 1}, \varphi_{\bm x,2}, \varphi_{\bm x, 3}$, respectively. 
The benefit of these degrees of freedom is shown by the following lemma, indicating that the resulting basis functions are decoupled mutually.

\begin{lemma}
    For $j = 0,1,2,3$, the following sequence
\begin{equation}\label{eq:3D:hess-sk-com} 
\begin{split}
P_1 \xrightarrow[]{\subset} \Span\{ \varphi_{\bm x, j} : \bm x \in \mathsf V\} \xrightarrow[]{\hess} & \Span\{\varphi_{\bm e, j} : \bm e \in \mathsf E\} \xrightarrow[]{\operatorname{curl}} \\ & \Span\{\varphi_{\bm f, j}: \bm f \in \mathsf F \} 
\xrightarrow[]{\operatorname{div}} \Span\{\varphi_{\bm K, j}: \bm K \in \mathsf K \}\xrightarrow[]{} 0
\end{split}
\end{equation}
is a complex.
\end{lemma}
\begin{proof}
It follows from \eqref{eq:int-by-part-hessu} that the Hessian of the space $\Span\{ \varphi_{\bm x, j} : \bm x \in \mathsf V\}$ is a subspace of the space $\Span\{\varphi_{\bm e, j} : \bm e \in \mathsf E\}$. The identities 
$$(\curl \varphi_{\bm e, j}\bm n, \bx)_{\bm f} = (\rot_{\bm f}\mathbb{E}_{\bm f} \varphi_{\bm e, j}, E_{\bm f}\bx)_{\bm f}  = (\varphi_{\bm e, j}\bm t, \bx)_{\partial \bm f}$$ for $j = 0$
and 
$$(\curl \varphi_{\bm e, j} \bm n, \be_j)_{\bm f} = (\rot_{\bm f}{E}_{\bm f}\varphi_{\bm e, j},  \be_j)_{\bm f} = (\varphi_{\bm e, j}\bm t, \be_j)_{\partial \bm f}$$
for $j = 1,2,3$, together with \Cref{prop:3D:hess-sk}, yield that $\curl \varphi_{\bm e, j} \in \Span\{\varphi_{\bm f, j}: \bm f \in \mathsf F \}$. The identities
$$(\div \varphi_{\bm e, j}, \bx)_{\bm K} = - (\varphi_{\bm e, j}, \dev\grad \bx)_{\bm K} + (\varphi_{\bm e, j}\bm n, \bx)_{\partial \bm K} = (\varphi_{\bm e, j}\bm n, \bx)_{\partial  \bm K}$$
for $k= 0$
and $$(\div \varphi_{\bm e, j}, \be_j)_{\bm K} = - (\varphi_{\bm e, j}, \dev\grad \be_j)_{ \bm K} + (\varphi_{\bm e, j}\bm n, \be_j)_{\partial \bm K} = (\varphi_{\bm e, j}\bm n, \be_{j})_{\partial \bm K}$$
for $j = 1,2,3$, together with \Cref{prop:3D:hess-sk}, yield that $\div \varphi_{\bm f, j} \in \Span\{\varphi_{\bm K, j}: \bm f \in \mathsf F \}$. Here $\dev \bm \xi := \bm \xi - \frac{1}{3} \operatorname{tr}(\bm \xi)I$ takes the traceless part for a matrix-valued function $\bm \xi$. A summary of the above arguments completes the proof.
\end{proof} 

The motivation for extracting the skeletal complex is to treat the local bounded commuting projection to each component separately. To prove the following proposition, we rely heavily on the structure of the gradgrad complex, as discussed in \Cref{sec:proof:weight-3d} below.

\begin{proposition}\label{prop:weight-3d}
    There exist weights $z_{\bm x, j} \in L^2(\omega_{\bm x})$ for $\bm x \in \mathsf V$, $z_{\bm e, j} \in L^2(\omega_{\bm e}^h; \bS)$ for $\bm e \in \mathsf E$, $z_{\bm f, j} \in L^2(\omega_{\bm f}^h; \mathbb T)$ for $\bm f \in \mathsf F$, and $z_{\bm K, j} \in L^2(\omega_{\bm K}^h ;\mathbb R^3)$ for $\bm K \in \mathsf K$, for $j = 0,1,2,3$, such that: 
    \begin{enumerate}
        \item
            For $u \in U_h(\omega_{\bm x})$, it holds that 
            $$(u, z_{\bm x, 0})_{\omega_{\bm x}} = (\bx \cdot \nabla u)(\bm x)- u(\bm x), (u, z_{\bm x, 1})_{\omega_{\bm x}} = \frac{\partial}{\partial x}u(\bm x), $$ $$(u, z_{\bm x, 2})_{\omega_{\bm x}} = \frac{\partial}{\partial y}u(\bm x),\text{ and }(u, z_{\bm x, 3})_{\omega_{\bm x}} = \frac{\partial}{\partial z}u(\bm x).$$
        \item For each $j$, let 
        $$\mathcal M^0_{j}u = \sum_{\bm x}(u, z_{\bm x, j})_{\omega_{\bm x}} \varphi_{\bm x, j}, \mathcal M^1_j\bm \sigma = \sum_{\bm x} (\bm \sigma, z_{\bm e, j})_{\omega_{e}^h} \varphi_{\bm e, j},$$ 
        $$\mathcal M^2_j \bm v = \sum_{\bm f} (\bm v, z_{\bm f, j})_{\omega_{\bm f}^h} \varphi_{\bm f, j},\text{ and }\mathcal M^3_j \bm q = \sum_{\bm K} (\bm q, z_{\bm K, j})_{\omega_{\bm K}^h} \varphi_{\bm K, j}$$ for $u \in L^2(\Omega), \bm \sigma \in L^2(\Omega; \bS), \bm v \in L^2(\Omega; \mathbb T), \bm q \in L^2(\Omega; \mathbb R^3)$, respectively, then for $u \in H^2(\Omega), \bm \sigma \in H(\curl, \Omega; \bS), \bm v \in H(\div, \Omega; \mathbb T)$, it holds that $$\cM^1_j \hess u = \hess \cM^0_j u, \cM^2_j \curl \bm \sigma = \curl \cM^1_j \bm \sigma, \cM^3_j \div \bm v = \div \cM^2_j \bm v$$ for $j = 0,1,2,3$.
    \end{enumerate}
    As a result, define $\cM^l = \sum_{j = 0}^3 \cM^l_j$, $l = 0,1,2,3$. Then it holds that $\cM^1 \hess u = \hess \cM^0 u$, $\cM^2 \rot \bm \sigma = \rot \cM^1 \bm \sigma$ and $\cM^3 \div \bm v = \div \cM^2 \bm v.$ Moreover, $\cM^0$ is a projection operator from $L^2(\Omega)$ to $\Span\{ \varphi_{\bm x, j'}, \bm x \in \mathsf V, j' = 0,1,2,3\}$.
\end{proposition}
The proof will be presented in \Cref{sec:proof:weight-3d}.
\begin{remark}
    Here $\sum_{\bm x}, \sum_{\bm e}, \sum_{\bm f}$ and $\sum_{\bm K}$  are the abbreviations of $\sum_{\bm x \in \mathsf V}$, $\sum_{\bm e \in \mathsf E}$, $\sum_{\bm f \in \mathsf F}$, $\sum_{\bm K \in \mathsf K}$, respectively.
    \end{remark}
On the other hand, once the projection operators for the skeletal complex are constructed, then the other imposed degrees of freedom can be treated in a rather systematic approach, which only requires the exactness of the finite element gradgrad complex. The complete proof is displayed in \Cref{sec:complete-proof}.
\subsection{Proof of \Cref{prop:3D:hess-sk}}
\label{sec:proof:hess-sk}

This subsection proves \Cref{prop:3D:hess-sk}. To this end, consider the following bubble complexes on subsimplices:
on edge $\bm e$, the bubble complex is
\begin{equation}
    \label{eq:edgebubble-1}
   0 \xrightarrow{} (\lambda_0\lambda_1)^5P_{k-8}(\bm e) \xrightarrow[]{d^2/d\bm t^2} (\lambda_0 \lambda_1)^3 P_{k-6}(\bm e)/ P_1(\bm e) \xrightarrow{} 0,
\end{equation}
and 
\begin{equation}
    \label{eq:edgebubble-2}
    0 \xrightarrow{} (\lambda_0\lambda_1)^4P_{k-7}(\bm e) \xrightarrow[]{d/d\bm t} (\lambda_0\lambda_1)^3P_{k-6}({e})/P_0({e}) \xrightarrow{} 0;
\end{equation}
on face $\bm f$, the bubble complex is
\begin{equation}
    \label{eq:facebubble-1}
    0 \xrightarrow{}(\lambda_0\lambda_1\lambda_2)^3P_{k-7}(\bm f)  \xrightarrow[]{\hess_{\bm f}}  (\lambda_0\lambda_1\lambda_2) P^{(0)}_{k-3}(\bm f; \mathbb S_{2\times 2}) \xrightarrow[]{\operatorname{rot}_{\bm f}}  P_{k-1}^{(1)}(\bm f;\mathbb R^2) / RT(\bm f) \xrightarrow{} 0,
\end{equation}
and 
\begin{equation}
    \label{eq:facebubble-2}
    0 \xrightarrow{}(\lambda_0\lambda_1\lambda_2)^2P_{k-5}(\bm f) \xrightarrow[]{\nabla_{\bm f}}  (\lambda_0\lambda_1\lambda_2) P^{(0)}_{k-3}(\bm f;\mathbb R^2) \xrightarrow[]{\operatorname{rot}_{\bm f}} P_{k-1}^{(1)}(\bm f)/\mathbb R \xrightarrow{} 0;
\end{equation}
in element $\bm K$, the bubble complex is
\begin{equation}
	\label{eq:elebubble}
    0 \xrightarrow{} (\lambda_0\lambda_1\lambda_2\lambda_3)^2P_{k-6}(\bm K) \xrightarrow[]{\hess}  B_{\bm K, k}^{\curl;\mathbb S} \xrightarrow[]{\operatorname{curl}}B_{\bm K, k-1}^{\div;\mathbb T}\xrightarrow[]{\operatorname{div}}  P_{k-2}^{(0)}(\bm K;\mathbb R^3) / RT(\bm K)\xrightarrow{} 0.
\end{equation}

Here the space $P^{(0)}_{k-3}$ is defined in \eqref{eq:P0}, the space $P_{k-1}^{(1)}$ is defined in \eqref{eq:P1}, the space $B_{\bm K, k}^{\curl;\mathbb S}$ is defined in \eqref{eq:BcurlS} and the space $B_{\bm K, k-1}^{\div;\mathbb T}$ is defined in \eqref{eq:BdivT}, respectively.
In what follows, the exactness of the above bubble complexes will be proved.

\begin{lemma}
    The polynomial sequences \eqref{eq:edgebubble-1} and \eqref{eq:edgebubble-2} are exact complexes.
\end{lemma}
\begin{proof}
Suppose that $u \in (\lambda_0\lambda_1)^5 P_{k-8}(\bm e)$, clearly $u''  := \partial ^2 u / \partial \bm t^2 \in (\lambda_0\lambda_1)^3 P_{k-6}(\bm e)$, it then suffices to show that $\int_e u''q = 0$ for all $q \in P_1(\bm e)$. A simple calculation yields that 
$\int_e u'' q = \int u q'' = 0.$ Since $\dim P_{k-8}(\bm e) = \dim P_{k-6}(\bm e) - \dim P_1(\bm e)$, it then follows that \eqref{eq:edgebubble-1} is exact. Similarly, it is not difficult to show that \eqref{eq:edgebubble-2} is also exact.
\end{proof}

As for the face bubble complex, first show that \eqref{eq:facebubble-2} is an exact complex. 
\begin{lemma}
    \label{lem:exact-facebubble-2-3d}
Suppose that $k \ge 7$, then \eqref{eq:facebubble-2} is an exact complex.
\end{lemma}
\begin{proof}
Suppose that $u = (\lambda_0\lambda_1\lambda_2)^2 u_1$ for some $u_1 \in P_{k-5}(\bm f)$. Then $$\nabla_{\bm f} u = (\lambda_0\lambda_1\lambda_2) (\sum_{i} \lambda_{i+1}\lambda_{i+2}u_1\nabla_{\bm f} \lambda_i + \lambda_0\lambda_1\lambda_2 \nabla_{\bm f} u_1),$$ clearly, $(\sum_{i} \lambda_{i+1}\lambda_{i+2}u_1\nabla_{\bm f} \lambda_i + \lambda_0\lambda_1\lambda_2 \nabla_{\bm f} u_1) \in P_{k-3}^{(0)}(\bm f; \mathbb R^2).$

Now suppose that $\bm w = (\lambda_0\lambda_1\lambda_2) \bm w_1$ with $\bm w_1 \in   P_{k-3}^{(0)}(\bm f; \mathbb R^2)$. Then $$\rot_{\bm f} \bm w = \curl_{\bm f}(\lambda_0\lambda_1\lambda_2) \cdot\bm w_1 + (\lambda_0\lambda_1\lambda_2) \rot_{\bm f} \bm w_1.$$

To prove that $\rot_{\bm f} \bm w \in P_{k-1}^{(1)}(\bm f)$, it suffices to show $\curl_{\bm f}(\lambda_0\lambda_1\lambda_2) \cdot \bm w_1 \in P_{k-1}^{(1)}(\bm f)$ since $(\lambda_0\lambda_1\lambda_2)$ is already in $P_{k-1}^{(1)}(\bm f)$. Indeed, this comes from the fact that $\curl_{\bm f} (\lambda_0\lambda_1\lambda_2) \in P_2^{(0)}(\bm f;\mathbb R^2)$, and $\bm w_1 \in P_{k-3}^{(0)}(\bm f)$ that their product is in $P_{k-1}^{(1)}(\bm f)$. Moreover, $\int_{f} \rot_{\bm f} \bm w = \int_{\partial f} \bm w\cdot \bm t = 0$. Consequently, \eqref{eq:facebubble-2} is a complex.

Since 
\begin{equation*}
\begin{split}
\dim P_{k-5}(\bm f) + \dim P_{k-1}^{(1)}(\bm f) - 1 = &\frac{1}{2} (k-4)(k-3) + \frac{1}{2} (k)(k+1) - 9 - 1 \\ =&  k^2 - 3k - 4 = (k-2)(k-1) - 2 \times 3 \\ =& 2\dim P_{k-3}^{(0)}(\bm f),
\end{split}
\end{equation*}
it remains to show that if $\bm \xi \in (\lambda_0\lambda_1\lambda_2)P_{k-3}^{(0)}(\bm f;\mathbb R^2)$ such that $\rot_{\bm f} \bm \xi = 0$, then there exists $u \in (\lambda_0\lambda_1\lambda_2)^2 P_{k-5}(\bm f)$ such that $\bm \xi= \grad_{\bm f} u$. By the exactness of the polynomial complex, there exists $u \in P_{k-1}(\bm f)$ such that $\bm \xi = \grad_{\bm f} u$. Assume that $u(\bm x) = 0$ for some vertex $\bm x$ of $\bm f$. Since $\grad_{\bm f} u = \bm \xi$ vanishes on the boundary, it then indicates that $u$ vanishes on the boundary. Therefore, it holds that $u \in  (\lambda_0\lambda_1\lambda_2)^2 P_{k-5}(\bm f)$, which completes the proof.
\end{proof}


\begin{lemma}
    \label{lem:exact-facebubble-1-3d}
    Suppose that $k \ge 7$, then \eqref{eq:facebubble-1} is an exact complex.
\end{lemma}
\begin{proof}
Suppose that $ u = (\lambda_0\lambda_1\lambda_2)^3 u_1$ with $ u_1 \in P_{k-7}(\bm f)$. Then 
$$\hess_{\bm f} u = b_{\bm f}\big\{6 \grad_{\bm f} b_{\bm f} (\grad_{\bm f} b_{\bm f})^T u_1 + 3b_{\bm f} \hess_{\bm f} b_{\bm f} u_1 + 3 b_{\bm f} [(\grad_{\bm f} b_{\bm f})(\grad_{\bm f} u)^T + (\grad_{\bm f} u)(\grad_{\bm f} b_{\bm f})^T] + b_{\bm f}^2 \hess_{\bm f} u\big\},$$
where $b_{\bm f} = \lambda_0\lambda_1\lambda_2$. Clearly it holds that $\hess_{\bm f} u \in (\lambda_0\lambda_1\lambda_2) P_{k-3}^{(0)}(\bm f; \mathbb S_{2\times 2})$.

Suppose that $\bm \sigma = (\lambda_0\lambda_1\lambda_2) \bm \sigma_1$, where $\bm \sigma_1 \in P_{k-3}^{(0)}(\bm f; \bS_{2\times 2}),$ then 
$$\rot_{\bm f} \bm \sigma = \bm \sigma_1 \cdot \curl (\lambda_0\lambda_1\lambda_2) + (\lambda_0\lambda_1\lambda_2) \rot_{\bm f} \bm \sigma_1.$$
A similar argument shows that $\bm \sigma_{1} \in P_{k-3}^{(0)}(\bm f; \mathbb S_{2\times 2})$ implies $\rot_{\bm f}\bm \sigma \in P_{k-1}^{(1)}(\bm f;\mathbb R^2)$. 
For $\bm \xi \in {RT}(\bm f)$, it follows from $\sym \curl_{\bm f} \bm \xi = 0$ that $\int_{f} \rot_{\bm f} \bm \sigma \cdot \bm \xi = \int_{f} \bm \sigma \cdot \sym \curl_{\bm f} \bm \xi = 0$. 

Since 
\begin{equation}
\begin{split}
\dim P_{k-7}(\bm f)+ 2(\dim P_{k-1}^{(1)}(\bm f)) - 3 = & \frac{1}{2}(k-6)(k-5) + \frac{2}{2}k(k+1) - 21 \\  = & \frac{1}{2}(3k^2 - 9k -12) = 3(\frac{1}{2}(k-2)(k-1) - 3) \\  =& 3\dim P_{k-3}^{(0)}(\bm f),
\end{split}
\end{equation}
it remains to show that if $ \bm \sigma \in (\lambda_0\lambda_1\lambda_2) P_{k-3}^{(0)}(\bm f; \mathbb S_{2\times 2})$, such that $\rot_{\bm f} \bm \sigma = 0$, then there exists $u \in (\lambda_0\lambda_1\lambda_2)^3P_{k-7}(\bm f)$ with $\bm \sigma = \hess_{\bm f} u$. By the exactness of the polynomial complex, such that $u$ can be found in $P_{k+2}(\bm f)$. Moreover, assume that $u(\bm x) = 0$ and $\grad_{\bm f} u(\bm x) = 0$ at some vertex $\bm x$ of $\bm f$. Since $\nabla_{\bm f}^2 u$ vanishes on $\partial \bm f$, it then implies that $\grad_{\bm f} u$ and $u$ vanish on $\partial \bm f$. It then concludes that \eqref{eq:facebubble-1} is an exact sequence. 
\end{proof}

Lastly, we prove the exactness of the sequence in \eqref{eq:elebubble}.

\begin{lemma}
\label{lem:exact-elebubble-3d}
Suppose that $k \ge 7$, then the sequence in \eqref{eq:elebubble} is an exact complex.
\end{lemma}
\begin{proof}
First, it follows from \cite[Theorem 4.8]{2021HuLiang}, the operator $$\operatorname{div}: B_{\bm K,k-1}^{\div;\mathbb T}\rightarrow P_{k-2}^{(0)}(\bm K;\mathbb R^3) / RT(\bm K)$$ is surjective. For any $\bm{\sigma}\in B_{\bm K, k}^{\curl;\mathbb S}$ with $\operatorname{curl}\bm{\sigma}=0$, by the exactness of the polynomial complex, there exists a $u\in P_{k+2}(\bm K)$ such that $\hess u =\bm{\sigma}$. Moreover, assume that $u(\bm{x})$ and $\nabla u(\bm{x})$ vanish at the given vertex $\bm{x}$ of $\bm K$. Then since $\hess u\times \bm{n} = \bm \sigma\times\bm{n}$ vanishes on all faces of $\bm K$, it follows that $u$ and $\nabla u$ vanish on all the faces of $\bm K$, which implies $u \in (\lambda_0\lambda_1\lambda_2\lambda_3)^2P_{k-6}(\bm K)$. 

Finally, it suffices to check the dimensions. Note that  $$\dim (\lambda_0\lambda_1\lambda_2\lambda_3)^2P_{k-6}(\bm K)= \frac{1}{6}(k-5)(k-4)(k-3),$$ and $$\dim P_{k-2}^{(0)}(\bm K;\mathbb R^3) / {RT}(\bm K) =  3(\frac{1}{6}(k-1)k(k+1)-4)-4.$$ It follows from \cite[Section 3.2]{2021HuLiang} that 
$$\dim B_{\bm K, k}^{\curl;\mathbb S} = k^3-4k^2+5k-14,$$ and from \cite[Section 4.2]{2021HuLiang} that $$\dim B_{\bm K,k-1}^{\div;\mathbb T} =(4k^3-6k^2-10k-60)/3.$$ By a direct calculation, it holds that
\begin{equation*}
	\operatorname{dim}(\lambda_0\lambda_1\lambda_2\lambda_3)^2P_{k-6}(\bm K) -\operatorname{dim}B_{\bm K, k}^{\curl,\mathbb S}+ \operatorname{dim}B_{\bm K,k-1}^{\div;\mathbb T}- \operatorname{dim}P_{k-2}^{(0)}(\bm K;\mathbb R^3 ) / {RT}(\bm K) =0.
\end{equation*}
Hence, the bubble complex \eqref{eq:elebubble} is exact.
\end{proof}

Now, it is ready to prove \Cref{prop:3D:hess-sk}.

\begin{proof}[Proof of \Cref{prop:3D:hess-sk}] The proof is separated into three parts.

\paragraph{\textbf{Step 1.}} Suppose that $u \in U_h$ vanishes at all but the first set of degrees of freedom \eqref{dof:4a} of the $H^2$ conforming finite element space $U_h$, set $\bm \sigma = \hess u$. Clearly, $\bm \sigma$ vanishes at the second and the fifth sets of the degrees of freedom for the $H(\curl;\mathbb S)$ conforming finite element space $\Sigma_h$, namely, \eqref{dof:5b} and \eqref{dof:5e}. To show that $\bm \sigma$ vanishes at the third set of degrees of freedom \eqref{dof:5c}, it suffices to notice that from the exactness of \eqref{eq:edgebubble-1}, for $p \in (\lambda_0\lambda_1)^3P_{k-6}(\bm e)/ P_1(\bm e)$, there exists $b \in (\lambda_0\lambda_1)^5P_{k-8}(\bm e)$ such that $b'' = p$. Then the third set of degrees of freedom of the $H^2$ conforming finite element space, namely, \eqref{dof:4c} implies that $\bm \sigma$ vanishes at \eqref{dof:5c}. Similarly, the exactness of \eqref{eq:edgebubble-2} and the fact that $u$ vanishes at \eqref{dof:4d} indicate that $\bm \sigma$ vanishes at \eqref{dof:5d}. 

It remains to show that $\bm \sigma$ vanishes at \eqref{dof:5f}, \eqref{dof:5g} and \eqref{dof:5h}, whose arguments are similar to each other. Take \eqref{dof:5f} as an example. From the exactness of the face bubble complex \eqref{eq:facebubble-1}, it holds that for $\bm \eta \in (\lambda_0\lambda_1\lambda_2)P_{k-3}^{(0)}(\bm f; \mathbb S_{2\times 2})$, there exists $$p \in B_{\bm f, k-7}^2 \text{ such that }\hess_{\bm f} p = \mathcal P_{\hess_{\bm f} B_{\bm f,k-7}^2} \bm \eta.$$ Since $\mathbb E_{\bm f}\bm \sigma = \hess_{\bm f}u $, the inner product \eqref{eq:inner-5f} now becomes $(\hess_{\bm f} u, \hess_{\bm f} p)_{\bm f} = 0$. Similarly, it can be shown that $\bm \sigma$ vanishes at \eqref{dof:5g} and \eqref{dof:5h}. In summary, $\bm \sigma$ vanishes at all degrees of freedom of the $H(\curl;\mathbb S)$ conforming finite element space $\bm \Sigma_h$ but \eqref{dof:5a}.

~

\paragraph{\textbf{Step 2.}} Suppose that $\bm \sigma \in \bm \Sigma_h$ vanishes at all but the first set of degrees of freedom \eqref{dof:5a} of $\bm \Sigma_h$. Set $\bm v = \curl \bm \sigma$, and clearly $\bm v$ vanishes at \eqref{dof:6b}. The rest of verification is based on the vector identities $$E_{\bm f} (\curl \bm \sigma \bm n) = \rot_{\bm f} \mathbb E_{\bm f} \bm \sigma, \text{ and }\bm n^T \curl \bm \sigma \bm n = \rot_{\bm f} E_{\bm f} (\bm \sigma \bm n),$$ and the exactness of \eqref{eq:facebubble-1} and \eqref{eq:facebubble-2}. To show that $\bm v$ vanishes at \eqref{dof:6c}, it suffices to notice that by the exactness of \eqref{eq:facebubble-1}, for $\bm w \in P_{k-1}^{(1)}(\bm f; \mathbb R^2)/RT(\bm f)$ there exists a function $$\bm \eta \in (\lambda_0\lambda_1\lambda_2) P_{k-3}^{(0)}(\bm f; \mathbb S_{2\times 2}) / \hess_{\bm f} B_{\bm f,k-7}^2 \text{ such that } \rot_{\bm f} \bm \eta = \bm w.$$ As a result, 
$$(E_{\bm f}(\bm v \bm n), \bm w)_{\bm f} = (\rot_{\bm f} \mathbb E_{\bm f} \bm \sigma, \rot_{\bm f} \bm \eta) = (\rot_{\bm f} \mathbb E_{\bm f} \bm \sigma, \rot_{\bm f} \bm \eta) + (\mathcal{P}_{\hess_{\bm f} B_{\bm f, k-7}^2}\mathbb{E}_{\bm f}\bm{\sigma},\mathcal{P}_{\hess_{\bm f} B_{\bm f, k-7}^2} \bm \eta)_{\bm f}.$$
Then it follows from \eqref{dof:5f} that $\bm v$ vanishes at \eqref{dof:6c}. A similar argument can show that $\bm v$ vanishes at \eqref{dof:6d} and \eqref{dof:6e}.

~

\paragraph{\textbf{Step 3.}} Finally, suppose that $\bm v \in \bm V_h$ vanishes at all but the first set of degrees of freedom \eqref{dof:6a} of $\bm V_h$, set $\bm q = \div \bm v$. Clearly, $\bm q$ vanishes at \eqref{dof:7b}. To show it also vanishes at \eqref{dof:7c}, note that the exactness of \eqref{eq:elebubble} indicates that for $\bm p \in P_{k-2}^{(0)}(\bm K; \mathbb R^3) / RT(\bm K)$ there exists $\bm \xi \in B_{\bm K,k-1}^{\div; \mathbb T} / \curl B_{\bm K, k-1}^{\curl, \mathbb S}$ such that $\div \bm \xi = \bm p$. Consequently, 
$$(\bm q, \bm p)_{\bm K} = (\div \bm v, \div \bm \xi)_{\bm K} = (\div \bm v, \div \bm \xi)_{\bm K}+ (\mathcal P_{\curl B_{\bm K,k}^{\curl;\mathbb S}} \bm v,\mathcal P_{\curl B_{\bm K,k}^{\curl;\mathbb S}} \bm \xi)_{\bm K}.$$
Therefore, \eqref{dof:6e} implies that $\bm q$ vanishes at \eqref{dof:7c}, which completes the proof.
\end{proof}

\subsection{Proof of \Cref{prop:weight-3d}}
    \label{sec:proof:weight-3d}
This subsection is the most technical part of this paper. Several constructions of such weight functions like $z_{\sigma, j}$ herein of the de Rham case were proposed in the literature. For example, Falk and Winther \cite{2015FalkWinther} used the double complex structure on their construction. This is later extended to the more general cases including other discrete de Rham subcomplexes, see \cite{2023HuLiangLin}. However, the approach therein is not suitable for the following construction, since it requires a special finite element dual complex (in our context, div-div complex) with homogeneous boundary conditions. Here we extend the construction of Arnold and Guzm\'{a}n \cite{2021ArnoldGuzman} to the finite element gradgrad complex.

Recall the following lemma from \cite[Theorem 3.1]{2022PaulySchomburg}, which asserts the existence of the Bogovskii-like operators for the divdiv complex.
\begin{lemma}
    \label{lemma:divdiv-3d}
    The following results hold for a contractible Lipschitz domain $\omega$ and a positive integer $s$.
    \begin{enumerate}
    \item Suppose that $p \in H_0^s(\omega)$ such that $p \perp P_1(\omega)$, then there exists $\bm \sigma \in H_0^{s+2}(\omega; \mathbb S)$ such that $\div \div \bm \sigma = p$. 
    \item Suppose that $\bm \sigma \in H_0^s(\omega;\mathbb S)$ such that $\div \div \bm \sigma = 0$, then there exists $\bm v \in H_0^{s+1}(\omega; \mathbb T)$ such that $\sym\curl \bm v = \bm \sigma$.
    \item Suppose that $\bm v \in H_0^s(\omega; \mathbb T)$ such that $\sym \curl \bm v = 0$, then there exists $\bm u \in H_0^{s+1}(\omega; \mathbb R^3)$ such that $\bm v = \dev\grad \bm u.$
    \end{enumerate}
\end{lemma}
Here, the required higher regularity is a must in the sense that the trace of these Sobolev spaces can be defined locally, which is crucial in the following construction, see \Cref{rmk:high-reg}.


The operators $\mathcal M_j^l$, $j , l= 0,1,2,3$, are constructed sequentially.

    \subsubsection*{Construction of $\mathcal M_j^0$}
        The first step of the construction is to construct the weight functions $z_{\bm x, j}$ for $j = 0,1,2,3$, and $\bm x \in \mathsf V$. 
        Let $W_{\bm x, j} \in H_0^5(\omega_{\bm x}), j = 0,1,2,3$, be the function such that
        $$(W_{\bm x,0},p)_{\omega_{\bm x}} =  (\bx \cdot \nabla p)(\bm x)- p(\bm x),\quad \forall p \in P_1(\omega_{\bm x}),$$
        and 
        $$(W_{\bm x, j},p)_{\omega_{\bm x}} = \be_j \cdot \nabla p(\bm x), \quad \forall p \in P_1(\omega_{\bm x}), \text{ for } j = 1,2,3.$$
        
        Indeed, define $b \in H_0^1(\omega_{\bm x})$ as $b|_{\bm K} = \lambda_0\lambda_1\lambda_2\lambda_3$ for all $\bm K$ of $\omega_{\bm x}$, then in the space $b^5P_1(\omega_{\bm x})$, there are unique functions $W_{\bm x,j}$, $j = 0,1,2,3,$ solves the above systems, respectively.

        Let $\hat{Z}_{\bm x, j} \in U_h(\omega_{\bm x})/P_1(\omega_{\bm x})$, $j = 0,1,2,3$, solve the following systems
        $$\quad (b^5 \hess \hat{Z}_{\bm x, 0}, \hess u_h)_{\omega_{\bm x}} = - (u_h, W_{\bm x,0})_{\omega_{\bm x}} + \bx \cdot \nabla u_h(\bm x) - u_h(\bm x), \quad\forall u_h \in U_h(\omega_{\bm x}) \perp P_1(\omega_{\bm x}),$$
and
        $$\quad (b^5 \hess \hat{Z}_{\bm x, j}, \hess u_h)_{\omega_{\bm x}} = - (u_h, W_{\bm x, j})_{\omega_{\bm x}} + \be_j \cdot \nabla u_h(\bm x), \quad\forall u_h \in U_h(\omega_{\bm x}) \perp P_1(\omega_{\bm x}),$$
        for $j = 1,2,3$.

        Set $z_{\bm x, j} = W_{\bm x, j} + \div\div(b^5 \hess \hat{Z}_{\bm x, j}) \in H_0^3(\omega_{\bm x})$. 
         For a given $u \in H^2(\Omega)$, set $\mathcal M^0_j u = \sum_x (u, z_{\bm x, j})_{\omega_{\bm x}} \varphi_{\bm x, j}$. 
        For a given edge $\bm e$, a direct calculation like \eqref{eq:int-by-part-hessu} indicates that 
        $$((\hess \varphi_{\bm x, 0} )\bm t , \mathbf{x})_{\bm e} = \begin{cases} 1,  &\text{ if } \bm e = [\bm y, \bm x], \\ -1,& \text{ if } \bm e = [\bm x, \bm y].\end{cases}$$
        
        Therefore,
         \[\begin{split}
            \hess \cM^0_0 u & = \sum_{\bm x} (u, z_{\bm x, 0})_{\omega_{\bm x}} \nabla^2\varphi_{\bm x, 0} \\ 
           & = \sum_{\bm x} \sum_{\bm e} (u, z_{\bm x, 0})_{\omega_{\bm x}} (\nabla^2\varphi_{\bm x, 0}\bm t, \bx)_{\bm e} \varphi_{\bm e, 0} \\
             & = \sum_{\bm e}((u, z_{\bm b,0})_{\omega_{\bm b}} - (u, z_{\bm a,0})_{\omega_{\bm a}}) \varphi_{\bm e,0} \\
             & = \sum_{\bm e}(u, z_{\bm b,0} \vmathbb 1_{\omega_{\bm b}} - z_{\bm a,0} \vmathbb 1_{\omega_{\bm a}})_{\omega_{\bm e}^h} \varphi_{\bm e,0}.
         \end{split}\]
        Here $\bm e = [\bm a, \bm b]$.
         A similar calculation shows that for $j = 0,1,2, 3$, it holds that
        \begin{equation}\label{eq:hessM0k}\hess \cM^0_j u = \sum_{\bm e}(u, z_{\bm b,j} \vmathbb 1_{\omega_{\bm b}} - z_{\bm a,j} \vmathbb 1_{\omega_{\bm a}})_{\omega_{\bm e}^h} \varphi_{\bm e, j}.\end{equation} 
        
        \subsubsection*{Construction of $\mathcal M_j^1$}
        An integration by parts leads to $(\div\div(b^5 \hess \hat{Z}_{\bm x, j}), p)_{\omega_{\bm x}} = 0$ for all linear function $p$. As a result,
        \begin{equation}
        \begin{split}
            (z_{\bm b,j} \vmathbb 1_{\omega_{\bm b}} - z_{\bm a,j} \vmathbb 1_{\omega_{\bm a}}, p)_{\omega_e^h} &= (W_{\bm b,j}, p)_{\omega_{\bm b}} - (W_{\bm a,j}, p)_{\omega_{\bm a}} 
            \\ 
            & = \left\{\begin{aligned} (\bx \cdot \nabla p)(\bm b)- p(\bm b) - (\bx \cdot \nabla p)(\bm a) + p(\bm a), &\,\,\, j = 0;\\ \be_j \cdot \nabla p(\bm b) - \be_j \cdot \nabla p(\bm a), &\,\,\, j =1,2,3; \end{aligned} \right. \\ 
            &= 0,
        \end{split}
        \end{equation}
        where the last line comes from the fact that $p$ is linear.
        
        Consequently, by \Cref{lemma:divdiv-3d} there exists $z_{\bm e, j} \in H_0^5(\omega_{\bm e}^h;\mathbb{S})$ such that 
        $$\div \div z_{\bm e, j} = z_{\bm b,j} \vmathbb 1_{\omega_{\bm b}} - z_{\bm a,j} \vmathbb 1_{\omega_{\bm a}}.$$
 Therefore, \eqref{eq:hessM0k} indicates that 
        $$\hess \cM^0_j u = \sum_{\bm e} (u, \div \div z_{\bm e, j})_{\omega_{\bm e}^h} \varphi_{\bm e, j} = \sum_{\bm e} (\hess u, z_{\bm e, j})_{\omega_{\bm e}^{h}} \varphi_{\bm e, j}.$$
        
        Thus define $$\mathcal M^1_j\bm \sigma = \sum_{\bm e} (\bm \sigma, z_{\bm e, j})_{\omega_{\bm e}^{h}} \varphi_{\bm e, j} \text{ for } \bm \sigma \in L^2(\Omega; \mathbb S).$$
        
        \subsubsection*{Construction of $\mathcal M_j^2$}
        For a given face $\bm f$, since 
        \[\begin{split}(\curl \varphi_{[\bm a, \bm b],0} \bm n, \bx)_{\bm f} =  (\rot_{\bm f}\mathbb{E}_{\bm f}\varphi_{[\bm a,\bm b],0}, E_{\bm f}\bx)_{\bm f} 
        =   (\varphi_{[\bm a, \bm b],0}\bm t, \bx)_{\partial \bm f},\end{split}\] a direct calculation indicates that 
        $$(\curl \varphi_{[\bm a, \bm b], 0} \bm n  , \mathbf{x})_{\bm f} = \begin{cases} - 1 & \text{ if } \bm f = [\bm a, \bm b, \bm c] \text{ for some }\bm c, \\ 1 & \text{ if } \bm f = [\bm a, \bm c, \bm b] \text{ for some }\bm c, \\ 0 & \text{ otherwise.} \end{cases}$$

        Therefore,
         \[\begin{split}
            \curl \cM^1_0 \bm \sigma & = \sum_{\bm e} (\bm \sigma , z_{\bm e, 0})_{\omega_{\bm e}} \curl\varphi_{\bm e, 0} \\ 
           & = \sum_{\bm e} \sum_{\bm f} (\bm \sigma , z_{\bm x, 0})_{\omega_{\bm x}} (\curl \varphi_{\bm e, 0}\bm n, \bx)_{\bm f} \varphi_{\bm f, 0} \\
             & =  - \sum_{\bm f}(\bm \sigma , z_{[\bm a, \bm b],0} \vmathbb 1_{\omega_{[\bm a, \bm b]}^h}  + z_{[\bm b, \bm c],0} \vmathbb 1_{\omega_{[\bm b, \bm c]}^h} + z_{[\bm c, \bm a],0} \vmathbb 1_{\omega_{[\bm c, \bm a]}^h})_{\omega_{\bm f}^h} \varphi_{\bm f,0},
         \end{split}\]
        
        here $\bm f = [\bm a,\bm b,\bm c]$. Since 
        \begin{equation}
        \label{eq:divdivdeltaz}
            \begin{split}
            & \div\div \left( z_{[\bm a, \bm b],0} \vmathbb 1_{\omega_{[\bm a, \bm b]}^h}  + z_{[\bm b, \bm c],0} \vmathbb 1_{\omega_{[\bm b, \bm c]}^h} + z_{[\bm c, \bm a],0} \vmathbb 1_{\omega_{[\bm c, \bm a]}^h}\right)  \\ &= 
           ( z_{\bm b,0} \vmathbb 1_{\omega_{\bm b}} - z_{\bm a,0} \vmathbb 1_{\omega_{\bm a}}) + (z_{\bm c,0} \vmathbb 1_{\omega_{\bm c}} - z_{\bm b,0} \vmathbb 1_{\omega_{\bm b}}) + (z_{\bm a,0} \vmathbb 1_{\omega_{\bm a}} - z_{\bm c,0} \vmathbb 1_{\omega_{\bm c}}) \\ & = 0,
            \end{split}
        \end{equation}
        it follows from \Cref{lemma:divdiv-3d} that there exists $z_{\bm f, 0} \in H_0^6(\omega_{\bm f}^h;\mathbb{T})$ such that $$ - \sym \curl z_{\bm f, 0} = z_{[\bm a, \bm b],0} \vmathbb 1_{\omega_{[\bm a, \bm b]}^h}  + z_{[\bm b, \bm c],0} \vmathbb 1_{\omega_{[\bm b, \bm c]}^h} + z_{[\bm c, \bm a],0} \vmathbb 1_{\omega_{[\bm c, \bm a]}^h}.$$ As a result, it holds that 
        $$ \curl \cM^1_0 \bm \sigma = \sum_{\bm f}(\bm \sigma, \sym\curl z_{\bm f, 0})_{\omega_{\bm f}^h} \varphi_{\bm f,0} = \sum_{\bm f} (\curl \bm \sigma, z_{\bm f,0})_{\omega_{\bm f}^h} \varphi_{\bm f,0}.$$
        A similar argument shows the existence of $z_{\bm f, j}$ for $j = 1,2,3$. Thus define $$\mathcal M^2_j\bm v = \sum_{\bm f} (\bm v, z_{\bm f, j})_{\omega_{\bm f}^{h}} \varphi_{\bm f, j} \text{ for } \bm v \in L^2(\Omega; \mathbb T).$$ 
                It then follows that $\curl \mathcal M^1_{j}\bm \sigma = \mathcal M^2_j \curl \bm \sigma$ for $\bm \sigma \in H(\curl,\Omega;\mathbb S)$.

        \subsubsection*{Construction of $\mathcal M_j^3$}
        For a given face $\bm f$ of element $\bm K$, it holds that 
        \[
            (\div \varphi_{\bm f,0},\bx)_{\bm K} = - (\varphi_{\bm f,0}, \grad \bx)_{\bm K} + (\varphi_{\bm f,0}\bm n, \bx)_{\partial\bm K} = (\varphi_{\bm f,0}\bm n, \bx)_{\partial\bm K}
        ,          
        \]
        where the first term vanishes since $\varphi_{\bm f,0}$ is traceless. As a consequence, 
        \[
            \div \cM^2_0 \bm v 
              = \sum_{\bm K}(\bm v, z_{[\bm a, \bm b, \bm c],0} \vmathbb 1_{\omega_{[\bm a, \bm b, \bm c]}^h}  + z_{[\bm b, \bm c, \bm d],0} \vmathbb 1_{\omega_{[\bm b, \bm c, \bm d]}^h} + z_{[\bm c, \bm d, \bm a],0} \vmathbb 1_{\omega_{[\bm c, \bm d, \bm a]}^h} + z_{[\bm d, \bm a, \bm b],0} \vmathbb 1_{\omega_{[\bm d, \bm a, \bm b]}^h})_{\omega_{\bm K}^h} \varphi_{\bm K,0}.
         \]

         Similar to \eqref{eq:divdivdeltaz}, a direct calculation yields that $$\sym \curl (z_{[\bm a, \bm b, \bm c],0} \vmathbb 1_{\omega_{[\bm a, \bm b, \bm c]}^h}  + z_{[\bm b, \bm c, \bm d],0} \vmathbb 1_{\omega_{[\bm b, \bm c, \bm d]}^h} + z_{[\bm c, \bm d, \bm a],0} \vmathbb 1_{\omega_{[\bm c, \bm d, \bm a]}^h} + z_{[\bm d, \bm a, \bm b],0} \vmathbb 1_{\omega_{[\bm d, \bm a, \bm b]}^h}) = 0,$$
         it follows from \Cref{lemma:divdiv-3d} that there exists $z_{\bm K,0} \in H_0^7(\omega_{\bm K}^h; \mathbb R^3)$ such that 
         $$\dev \grad z_{\bm K,0} = z_{[\bm a, \bm b, \bm c],0} \vmathbb 1_{\omega_{[\bm a, \bm b, \bm c]}^h}  + z_{[\bm b, \bm c, \bm d],0} \vmathbb 1_{\omega_{[\bm b, \bm c, \bm d]}^h} + z_{[\bm c, \bm d, \bm a],0} \vmathbb 1_{\omega_{[\bm c, \bm d, \bm a]}^h} + z_{[\bm d, \bm a, \bm b],0} \vmathbb 1_{\omega_{[\bm d, \bm a, \bm b]}^h}.$$

         As a result, it holds that 
         $$ \div \cM^2_0 \bm v = \sum_{\bm f}(\bm v, \dev\grad z_{\bm K, 0})_{\omega_{\bm K}^h} \varphi_{\bm K,0} = \sum_{\bm f} (\div \bm v, z_{\bm f,0})_{\omega_{\bm K}^h} \varphi_{\bm K,0}.$$

         A similar argument shows the existence of $z_{\bm f, j}$ for $j = 1,2,3$. Thus define $$\mathcal M^3_j\bm q = \sum_{\bm K} (\bm q, z_{\bm K, j})_{\omega_{\bm K}^{h}} \varphi_{\bm K, j} \text{ for } \bm u \in L^2(\Omega; \mathbb{R}^3).$$

\begin{remark}
\label{rmk:high-reg}
The higher regularity (rather than $H(\div\div)$) is necessary to derive \eqref{eq:divdivdeltaz}, since the trace of functions in $H(\div\div, \Omega; \mathbb S)$ might be not globally well-defined.
\end{remark}

\subsection{The proof of \Cref{thm:main-3d}}
\label{sec:complete-proof}

The proof is similar to those in \cite{2023HuLiangLin}, with the help of harmonic inner products. 

Define the following operators $\mathcal R_{\omega}^l$ for $l = 0,1,2,3$:
\begin{itemize}\item[-] For $u \in H^2(\omega)$, define $\mathcal R_{\omega}^0 u $ 
    as the $L^2(\omega)$ projection to the space $P_1(\omega)$.

       \item[-]
   For $\bm \sigma \in H(\curl, \omega; \bS)$, define $\mathcal R^1_{\omega} \bm \sigma \in U_h(\omega)/ P_1 (\omega) $ via the following Galerkin projection:
   \begin{equation}
   (\hess \mathcal R_{\omega}^1 \bm \sigma, \hess v_h)_{\omega} = (\bm \sigma , \hess v_h)_{\omega}, \quad \forall v_h \in U_h(\omega).
   \end{equation}

\item[-]
For $\bm v \in H(\div, \omega; \mathbb T)$, define $\mathcal R_{\omega}^2 \bm v \in \bm \Sigma_h(\omega)$, such that $\mathcal R_{\omega}^2\bm v \perp \hess U_h(\omega)$ and   
\begin{equation}
(\curl \mathcal R_{\omega}^2\bm v , \curl \bm \sigma_h)_{\omega} = (\bm v, \curl \bm \sigma_h)_{\omega}, \,\, \forall \sigma_h \in \bm \Sigma_h(\omega).
\end{equation}
\item[-] For $\bm q \in L^2(\omega; \mathbb R^3)$, define $\mathcal R_{\omega}^3 \bm q \in \bm V_h(\omega)$ such that $\mathcal R_{\omega}^3\bm q \perp \curl \bm \Sigma_h(\omega)$, and 
\begin{equation}
    (\div \mathcal R_{\omega}^3\bm p , \div \bm v_h)_{\omega} = (\bm q, \div \bm v_h)_{\omega}, \,\, \forall v_h \in \bm V_h(\omega).
    \end{equation}

\end{itemize}

It follows from the exactness of the discrete gradgrad complex on patch $\omega$ that the operators $\mathcal R_{\omega}^l , l = 0,1,2,3$, are well-defined. As a result, the projection operators $\cQ_{\omega}^l, l = 0,1,2$, can be determined.
\begin{itemize}
    \item[-]
For $u \in H^2(\omega)$, define $\mathcal Q^0_{\omega} u = \cR_{\omega}^0 u + \mathcal R^1_{\omega} \hess u$, then it follows from the definition of $\cR_{\omega}^1$ that $\cQ_{\omega}^0 u $ is a projection on $\omega$.
\item[-] For $\bm \sigma \in H(\curl, \omega; \bS)$, define $\mathcal Q^1_{\omega} \bm \sigma = \hess \cR_{\omega}^1 \bm \sigma + \cR_{\omega}^2 \curl \bm \sigma.$ 
\item[-] For $\bm v \in H(\div, \omega; \mathbb T)$, define $\mathcal Q^2_{\omega} \bm v = \curl \cR_{\omega}^2 \bm v + \cR_{\omega}^3 \div \bm v.$
\end{itemize}

Note that $\cQ^1_{\omega} \bm \sigma $ is a projection onto $\bm \Sigma_h(\omega)$, $\cQ^2_{\omega}\bm v$ is a projection on $\bm V_h(\omega)$.
For convenience, given a simplex $\sigma$ and $l \ge 0$, let $\cQ_{\sigma}^l := \cQ_{\omega_{\sigma}}^l$ and $\cR_{\sigma}^l  := \cR_{\omega_{\sigma}}^l$; given $m > 0$, let $\cQ_{\sigma,[m]}^l := \cQ_{\omega_{\sigma}^{[m]}}^l$, $\cR_{\sigma,[m]}^l := \cR_{\omega_{\sigma}^{[m]}}^l$.

Let $L_{\sigma}^0$ ($L_{\sigma}^1, L_{\sigma}^2$, resp.) be the canonical interpolation operator defined by the degrees of freedom on the simplex (i.e., vertex, edge, face, element) $\sigma$ for the finite element space $U_h$ ($\bm \Sigma_h$, $\bm V_h$, resp.), except the first sets of degrees of freedoms of each finite element space, namely, \eqref{dof:4a}, \eqref{dof:5a}, and \eqref{dof:6a}. In what follows, we construct the projection operators sequentially. Note that the linear function $u \in P_1$ vanishes at $L_{\sigma}^0$. The value of $L_{\sigma}^{l}\eta$ is only depends on the value of $\eta$ on $\omega_{\sigma}$, and
the supports of $L_{\sigma}^{l}\eta$ are a subset of $\omega_{\sigma}$.

\subsubsection*{Construction of $\pi^0$}

It follows from the definition of $L_{\sigma}^0$ and \Cref{prop:weight-3d} that for $u \in U_h$,
\[ u 
  =  \sum_{\sigma} L_{\sigma}^0u  + \sum_{\bm x} \sum_{j = 0}^3 (u, z_{\bm x, j})_{\omega_{\bm x}} \varphi_{\bm x, j}  = \sum_{\sigma} L_{\sigma}^0u + \sum_{k=0}^3 \cM^0_j u  = \sum_{\sigma} L_{\sigma}^0 u+ \mathcal M^0 u.
\]
This motivates to define
\begin{equation}\label{eq:pi0u-2d}
       \pi^0 u = \sum_{\sigma} L_{\sigma}^0 \cQ_{\sigma}^0 u + \mathcal M^0 u.
\end{equation}
For $u \in U_h$, it holds that $\pi^0 u = u$, and the value of $\pi^0 u$ on $\omega_{\sigma}$ depends only on the value of $u$ on $\omega_{\sigma}^{[1]}.$

\subsubsection*{Construction of $\pi^1$}

Since $\pi^0 u = \sum_{\sigma} L_{\sigma}^0\cR_{\sigma}^1 \hess u + \cM^0 u,$ this motivates, for $\bm \sigma \in H(\curl, \Omega; \mathbb S)$, to define  
\begin{equation}
    \hat{\pi}^1 \bm \sigma = \sum_{\sigma} \hess L_{\sigma}^0 \cR_{\sigma}^1 \bm \sigma + \cM^1 \bm \sigma.
\end{equation}

Then it follows from \Cref{prop:weight-3d} that for $u \in H^2(\Omega)$, it holds that $
    \hat{\pi}^1 \hess u = \hess \pi^0 u.$
Note that $\hat{\pi}^1$ can be regarded as an operator from $H(\curl, \omega_{\sigma}^{[1]}; \bS)$ to $\bm \Sigma_h(\omega_{\sigma}^{[1]})$, when only the value of $\hat{\pi}^1 \bm \sigma$ on $\omega_{\sigma}$ is considered.

For convenience, denote by $\be_0 = \bx$. To get a projection operator, define the following modified interpolation 
\begin{equation}\label{eq:pi1-3d}
\pi^1 \bm \sigma = \hat{\pi}^1 \bm \sigma + \sum_{\sigma} L_{\sigma}^1 ( \id - \hat{\pi}^1) \mathcal Q_{\sigma,[1]}^1 \bm \sigma + \sum_{\bm e}\sum_{j = 0}^3 \Big( \big[( \id - \hat{\pi}^1)\mathcal Q_{\bm e,[1]}^1\bm\sigma\big]\bm t, \be_j\Big)_{\bm e} \varphi_{\bm e, j} .
\end{equation}

Since for $\bm \sigma \in \bm \Sigma_h$, it holds that $\cQ^1_{\sigma, [1]}\bm \sigma = \bm \sigma$ on $\omega_{\sigma}^{[1]}$, this leads to
$$\pi^1\bm \sigma = \hat{\pi}^1\bm \sigma + (\id - \hat\pi^1)\bm \sigma = \bm \sigma,$$
which implies that $\pi^1$ is a projection operator.

It follows from the definition of $\mathcal Q_{\omega}^1$ that 

\begin{equation}
    \label{eq:H21simplify}
    \begin{split} 
    ( \id - \hat{\pi}^1) \mathcal Q_{\sigma,[1]}^1 \bm \sigma = &  (\id - \hat{\pi}^1) \hess \mathcal R_{\sigma,[1]}^1 \bm \sigma + (\id - \hat \pi^1) \mathcal R_{\sigma,[1]}^2 \curl \bm \sigma\\
    = & \grad^2 (\id - \pi^0)  \mathcal R_{\sigma,[1]}^1 \bm \sigma + (\id - \hat \pi^1) \mathcal R_{\sigma,[1]}^2 \curl \bm \sigma \\ 
    = & (\id - \hat \pi^1) \mathcal R_{\sigma,[1]}^2  \curl \bm \sigma.
    \end{split} 
\end{equation}

The above two formulations show that $\pi^1$ is a projection operator and that it holds the commuting property $\pi^1 \hess u= \hess \pi^0 u$ for all $u \in H^2(\Omega).$

\subsubsection*{Construction of $\pi^2$} Consider the construction of $\pi^2$. 

Taking  $\curl$ on \eqref{eq:pi1-3d} yields that
\begin{equation}
\begin{split}     
\curl \pi^1 \bm \sigma = & \curl \cM^1 \bm \sigma +  \sum_{\sigma} \curl L_{\sigma}^1 ( \id - \hat{\pi}^1) \mathcal Q_{\sigma,[1]}^1 \bm \sigma + \sum_{\bm e}\sum_{j = 0}^3 \Big( \big[( \id - \hat{\pi}^1)\mathcal Q_{\bm e,[1]}^1\bm\sigma\big]\bm t, \be_j\Big)_{\bm e} \curl \varphi_{\bm e, j} \\
 = & \mathcal M^2 \curl \bm \sigma + \sum_{\sigma} \curl L_{\sigma}^1 ( \id - \hat{\pi}^1) \mathcal R_{\sigma,[1]}^2 \curl \bm \sigma + \sum_{\bm e}\sum_{j = 0}^3 \Big(\big[( \id - \hat{\pi}^1)\mathcal R_{\bm e,[1]}^2\curl\bm\sigma\big]\bm t, \be_j\Big)_{\bm e} \curl \varphi_{\bm e, j}.
\end{split}
\end{equation}

Inspired by the above identity, for $\bm v \in H(\div, \Omega; \mathbb T)$, define the following interpolation  
\begin{equation}
\hat{\pi}^2 \bm v =\mathcal M^2 \bm v + \sum_{\sigma} \curl L_{\sigma}^1 ( \id - \hat{\pi}^1) \mathcal R_{\sigma,[1]}^2 \bm v + \sum_{\bm e}\sum_{j = 0}^3 \Big(\big[( \id - \hat{\pi}^1)\mathcal R_{\bm e,[1]}^2\bm v\big]\bm t, \be_j\Big)_{\bm e} \curl \varphi_{\bm e, j} .
\end{equation}
Then it holds that $\hat\pi^2 \curl \bm\sigma = \curl \pi^1 \bm \sigma$ for all $\bm \sigma \in H(\curl,\Omega;\mathbb S)$. Similar to $\hat \pi^1$, $\hat{\pi}^2$ can be regarded as an operator from $H(\div, \omega_{\sigma}^{[2]}; \bT)$ to $\bm V_h(\omega_{\sigma}^{[2]})$, when only the value of $\hat{\pi}^2 \bm \sigma$ on $\omega_{\sigma}$ is considered. Similarly, this interpolation can be modified as follows, 
\begin{equation}\label{eq:pi2-3d}
\pi^2 \bm v = \hat{\pi}^2\bm v + \sum_{\sigma} L_{\sigma}^2 ( \id - \hat{\pi}^2) \mathcal Q_{\sigma,[2]}^2 \bm v +\sum_{\bm f}\sum_{k=0}^3\Big(\big[( \id - \hat{\pi}^2) \mathcal Q_{\sigma,[2]}^2 \bm v\big] \bm n, \be_j\Big)_{\bm f} \varphi_{\bm f, j}.
\end{equation}

Since for $\bm v \in \bm V_h$, it holds that $\mathcal Q_{\sigma,[2]}^2 \bm v = \bm v$ on $\omega_{\sigma}^{[2]}$,it leads to
$$ \pi^2 \bm v = \hat{\pi}^2\bm v + (\id - \hat\pi^2)\bm v = \bm v,$$ which indicates that $\pi^2$ is a projection operator. 

It follows from 
\begin{equation}
    \label{eq:H31simplify}
    \begin{split} 
    ( \id - \hat{\pi}^2) \mathcal Q_{\sigma,[2]}^2 \bm v = &  (\id - \hat{\pi}^2) \curl \mathcal R_{\sigma,[2]}^2 \bm v + (\id - \hat \pi^2) \mathcal R_{\sigma,[2]}^3 \div \bm v\\
    = & \curl (\id - \pi^1)  \mathcal R_{\sigma,[2]}^2 \bm v + (\id - \hat  \pi^2) \mathcal R_{\sigma,[2]}^3 \div \bm \sigma \\ 
    = & (\id - \hat \pi^2) \mathcal R_{\sigma,[2]}^3 \div \bm \sigma
    \end{split} 
\end{equation}
that $\pi^2$ satisfies the commuting property, namely, $\pi^2 \curl \bm \sigma = \hat\pi^2 \curl \bm \sigma = \curl \pi^1 \bm \sigma$ for $\bm \sigma \in H(\curl,\Omega; \mathbb S)$.

\subsubsection*{Construction of $\pi^3$}

Finally, it follows from $\div \hat{\pi}^2 \bm v = \div \mathcal M^2 v$ and \Cref{prop:weight-3d} that 
\begin{equation}
    \begin{split}
    \div \pi^2 \bm v & = \div\hat{\pi}^2 \bm v + \sum_{\sigma} \div L_{\sigma}^2 ( \id - \hat{\pi}^2) \mathcal Q_{\sigma,[2]}^2 \bm v  +\sum_{\bm f}\sum_{j=0}^3\Big(\big[( \id - \hat{\pi}^2) \mathcal Q_{\sigma,[2]}^2 \bm v\big] \bm n, \be_j\Big)_{\bm f} \div \varphi_{\bm f, j} \\ 
    & = \mathcal M^3 \div \bm v + \sum_{\sigma} \div L_{\sigma}^2 ( \id - \hat{\pi}^2) \mathcal R_{\sigma,[2]}^3 \div \bm v  +\sum_{\bm f}\sum_{j=0}^3\Big(\big[( \id - \hat{\pi}^2) \mathcal R_{\sigma,[2]}^3 \div \bm v\big]\bm n, \be_j\Big)_{\bm f} \div \varphi_{\bm f, j}.
    \end{split}
    \end{equation}

Thus, for $\bm q \in L^2(\Omega; \mathbb R^3)$, define
$$\pi^3 \bm q = \mathcal M^3 \bm q + \sum_{\sigma} \div L_{\sigma}^2 ( \id - \hat{\pi}^2) \mathcal R_{\sigma,[2]}^3 \bm q  +\sum_{\bm f}\sum_{j=0}^3\Big(\big[( \id - \hat{\pi}^2) \mathcal R_{\sigma,[2]}^3 \bm q\big]\bm n, \be_j\Big)_{\bm f} \div \varphi_{\bm f, j}.$$
Clearly it holds that $\pi^3 \div \bm v = \div \pi^2 \bm v$ for $\bm v \in H(\div,\Omega;\mathbb T)$.
Since $\div:\bm V_h \to \bm Q_h$ is surjective, for any $\bm q \in \bm Q_h$, there exists $\bm v \in \bm V_h$ such that $\div \bm v = \bm q$. Therefore, $\pi^3\bm q = \pi^3 \div \bm v = \div \pi^2 \bm v = \div \bm v = \bm q$, which completes the proof. The estimation of the local bounds is similar to that in \cite{2023HuLiangLin}.

\subsection{Remark: An abstract framework}
\label{sec:rmk}
In fact, the above argument can be generalized to the more general case. Here we present an abstract framework without a further proof. 

For the finite element space $U_h$, the (global) degrees of freedom consist of $\bx \cdot \nabla(\bm x) - u(\bm x)$, $\nabla u(\bm x)$ for $x \in \mathsf{V}$, and  $ f_{\sigma, i}^0, i = 1,2,\ldots, N_{\sigma}^0, \sigma \in \mathsf S$. 
Here $N_{\sigma}^0$ is the number the degrees of freedom (except $u(\bm x), \nabla u(\bm x)$ if $\sigma = \bm x$) attached to the sub-simplex $\sigma$ of the space $U_h$. 
The corresponding basis functions are then denoted as ${\varphi}_{\bm x,0}, {\varphi}_{\bm x,1}, {\varphi}_{\bm x,2}$, $\varphi_{\bm x, 3}$ and $\phi_{\sigma, i}^0$, respectively. 
Define $L_{\sigma}^0 = \sum_{i = 1}^{N_{\sigma}^0} f_{\sigma, i}^0(\cdot) \phi_{\sigma,i}^0$,
then by the definition of the basis functions and degrees of freedom, it holds that for $u \in U_h$,
\begin{equation}\label{eq:dofid-u3d}
	u = \sum_{\sigma} L_{\sigma}^0u + \sum_{\bm x} \big[ (\bx \cdot \nabla u(\bm x) - u(\bm x))\varphi_{\bm x,0} +\frac{\partial}{\partial x} u(\bm x) \varphi_{\bm x,1} + \frac{\partial}{\partial y} u(\bm x)\varphi_{\bm x,2} +  \frac{\partial}{\partial z} u(\bm x)\varphi_{\bm x,3}\big].
\end{equation}

For the finite element space $\bm \Sigma_h$, the (global) degrees of freedom consist of $(\bm \sigma\bm t, \bm p)_{\bm e}$ for $\bm p \in RT$ for $\bm e \in \mathsf{E}$, and $f_{\sigma,i}^1, i = 1,2,\cdots,N_{\sigma}^1, \sigma \in \mathsf{S}$. The corresponding basis functions are then denoted as $\varphi_{\bm e, j}$, $j = 0,1,2,3$, and $\phi_{\sigma,i}^1$. Define $L_{\sigma}^1 = \sum_{i = 1}^{N_{\sigma}^1} f_{\sigma, i}^1(\cdot) \phi_{\sigma,i}^1$ similarly. Then by definition of the basis functions and degrees of freedom, it holds that for $\bm \sigma \in \bm \Sigma_h$, 
\begin{equation}
    \bm \sigma = \sum_{\sigma} L_{\sigma}^1 \bm \sigma + \sum_{\bm e} \big[(\bm \sigma\bm t, \bx)_{\bm e} \varphi_{\bm e,0} + (\bm \sigma\bm t, \be_1)_{\bm e} \varphi_{\bm e,1} + (\bm \sigma\bm t, \be_2)_{\bm e} \varphi_{\bm e,2} + (\bm \sigma\bm t, \be_3)_{\bm e} \varphi_{\bm e,3}\big].
\end{equation}

For the finite element space $\bm V_h$, the (global) degrees of freedom consist of $(\bm v\bm n, \bm p)_{\bm f}$ for $\bm p \in RT$ for $\bm f \in \mathsf{F}$, and $f_{\sigma,i}^2, i = 1,2,\cdots,N_{\sigma}^1, \sigma \in \mathsf{S}$. The corresponding basis functions are then denoted as $\varphi_{\bm f, j}$, $j = 0,1,2,3$, and $\phi_{\sigma,i}^2$. Define $L_{\sigma}^2 = \sum_{i = 1}^{N_{\sigma}^2} f_{\sigma, i}^2(\cdot) \phi_{\sigma,i}^2$ similarly. Then by definition of the basis functions and degrees of freedom, it holds that for $\bm v \in \bm V_h$, 
\begin{equation}
    \bm v = \sum_{\sigma} L_{\sigma}^2 \bm v + \sum_{\bm f} \big[(\bm v\bm n, \bx)_{\bm f} \varphi_{\bm f ,0} + (\bm v\bm n, \be_1)_{\bm f} \varphi_{\bm f,1} + (\bm v\bm n, \be_2)_{\bm f} \varphi_{\bm f,2} +  (\bm v\bm n, \be_3)_{\bm f} \varphi_{\bm f,3} \big].
\end{equation}

For the finite element space $\bm Q_h$, the global degrees of freedom consist of $(\bm q, \bm p)_{\bm K}$ for $\bm p \in RT$, and $f_{\sigma,i}^3(\bm q)$, $p = 1,2,\cdots, N_{\sigma}^3, \sigma \in \mathsf S$.  The corresponding basis functions and degrees of freedom are then denoted by $\varphi_{\bm K, j}$, $j = 0,1,2,3$, and $\phi_{\sigma,i}^3$, Define $L_{\sigma}^3 = \sum_{i = 1}^{N_{\sigma}^3} f_{\sigma,i}^3(\cdot) \phi_{\sigma,i}^3.$ It holds that for $\bm q\in \bm Q_h$, 
\begin{equation}
\bm q = \sum_{\sigma} L_{\sigma}^3 \bm q + \sum_{\bm K} \big[ (\bm q,\bx)_{\bm K} \varphi_{\bm K,0} + (\bm q, \be_1)_{\bm K}\varphi_{\bm K,1} + (\bm q, \be_2)_{\bm K}\varphi_{\bm K,2} + (\bm q, \be_3)_{\bm K}\varphi_{\bm K,3}\big].    
\end{equation}

Now we can generalize our theorem to the finite element spaces $U_h \subseteq H^2(\Omega)$, $\bm \Sigma_h \subseteq H(\curl, \Omega; \mathbb S)$, $\bm V_h \subseteq H(\div, \Omega; \mathbb T)$ and $\bm Q_h \subseteq L^2(\Omega; \mathbb R^3)$, which satisfy the following Assumptions (A1)-(A4).

\begin{itemize}
\item[\bf (A1)] 
The finite element sequence
\begin{equation}\label{eq:3D:compl}
	P_1 \xrightarrow{\subset} U_h \xrightarrow[]{\hess} \bm \Sigma_h \xrightarrow[]{\curl} \bm V_h \xrightarrow[]{\operatorname{div}} \bm Q_h \xrightarrow{}0
\end{equation} is an exact complex. The exactness also holds when restricted on any (extended) patch $\omega =  \omega_{\sigma}^{h}$ or $\omega_{\sigma}^{[m]}$  for non-negative $m$ and $\sigma \in \mathsf S$, namely,
the sequence
\begin{equation}
	P_1 \xrightarrow{\subset} U_h(\omega) \xrightarrow[]{\hess} \bm \Sigma_h (\omega) \xrightarrow[]{\curl} \bm V_h(\omega) \xrightarrow[]{\operatorname{div}} \bm Q_h(\omega) \xrightarrow{}0,
\end{equation}
is exact, where $U_h(\omega), \bm \Sigma_h(\omega), \bm V_h(\omega), \bm Q_h(\omega)$ are the restrictions of $U_h, \bm \Sigma_h, \bm V_h, \bm Q_h$ on $\omega$, respectively.
\item[\bf (A2)] 
It holds that $\hess \varphi_{\bm x, j} \in \Span\{\varphi_{\bm e, j'}: \bm e \in \mathsf E: j' = 0,1,2,3\},$  $\curl \varphi_{\bm e, j} \in \Span\{\varphi_{\bm f, j'} :\bm f \in \mathsf F: j' = 0,1,2,3\}$, and $\div \varphi_{\bm f, j} \in \Span\{\varphi_{\bm K,j'} : j' = 0,1,2,3\}$ for $j = 0,1,2,3$. 
\item[\bf (A3)] The collection $L_{\sigma}^0$ of the degrees of freedom  vanishes for the linear functions, namely, $L_{\sigma}^0(u) = 0$ for $u \in P_1$ and $\sigma \in \mathsf S$.
\item[\bf (A4)] For $l = 0,1,2,3$, the mapping $L_{\sigma}^l$ is locally defined and locally supported: for $l = 0 $ the value of $L_{\sigma}^0 u $ on $\sigma$ only depends on the value of $u \in H^2(\Omega)$ on $\omega_{\sigma}$, and the support of $L_{\sigma}^0 u$ and $\varphi_{\bm x, j}$ are a subset of $\omega_{\sigma}$, for $j = 0,1,2,3$. Similar conditions hold for $\bm \Sigma_h (l = 1)$ and $\bm V_h (l = 2)$, and $\bm Q_h (l = 3)$.
 
\end{itemize}

\section{Two-dimensional case}
\label{sec:2d}
\subsection{The framework for two dimensions}

Instead of proving \Cref{thm:main-2d} directly, \Cref{prop:framework-2d} is proposed in the following, including more general cases of the finite element gradgrad complexes (and therefore elasticity complexes in two dimensions).

Simiar to those in three dimensions, suppose that 
\begin{itemize}
    \item[-] for the $H^2$ conforming finite element space $U_h$, the (global) degrees of freedom contain $u(\bm x)$, $\nabla u(\bm x)$ for $x \in \mathsf{V}$,
with corresponding basis functions denoted as $\tilde{\varphi}_{\bm x,0}, \tilde{\varphi}_{\bm x,1}, \tilde{\varphi}_{\bm x,2}$;
\item[-] for the $H(\rot;\mathbb S)$ conforming finite element space $\bm \Sigma_h$, the (global) degrees of freedom contain $(\bm \sigma\bm t, \bm w)_{\bm e}$ for $\bm w \in RT(\bm f)$ for $\bm e \in \mathsf{E}$, with corresponding basis functions denoted as $\varphi_{\bm e, 0}, \varphi_{\bm e, 1}, \varphi_{\bm e, 2}$;
\item[-] for the $L^2$ finite element space $\bm Q_h$, the global degrees of freedom consist of $(\bm q, \bm w)_{\bm e}$ for $\bm w \in RT(\bm f)$, with corresponding basis functions denoted as $\varphi_{\bm f, 0}, \varphi_{\bm f, 1}, \varphi_{\bm f, 2}$.

\end{itemize}
Similar to those notations in \Cref{sec:rmk}, define $L_{\sigma}^0$, $L_{\sigma}^1$, $L_{\sigma}^2$ accordingly, then it holds that for  $u \in U_h$,
\begin{equation}\label{eq:dofid-u}
	u = \sum_{\sigma} L_{\sigma}^0u + \sum_{\bm x} \big[ u(\bm x) \tilde \varphi_{\bm x,0} +\frac{\partial}{\partial x} u(\bm x) \tilde \varphi_{\bm x,1} + \frac{\partial}{\partial y} u(\bm x) \tilde \varphi_{\bm x,2}\big];
\end{equation}

for $\bm \sigma \in \bm \Sigma_h$, 
\begin{equation}
    \bm \sigma = \sum_{\sigma} L_{\sigma}^1 \bm \sigma + \sum_{\bm e} \big[(\bm \sigma\bm t, \bx)_{\bm e} \varphi_{\bm e,0} + (\bm \sigma\bm t, \be_1)_{\bm e} \varphi_{\bm e,1} + (\bm \sigma\bm t, \be_2)_{\bm e} \varphi_{\bm e,2} \big];
\end{equation}
for $\bm v\in \bm Q_h$, 
\begin{equation}
\bm v = \sum_{\sigma} L_{\sigma}^2 \bm v + \sum_{\bm f} \big[ (\bm v,\bx)_{\bm f} \varphi_{\bm f,0} + (\bm v, \be_1)_{\bm f}\varphi_{\bm f,1} + (\bm v, \be_2)_{\bm f}\varphi_{\bm f,2}\big].    
\end{equation}

Here are the assumptions of \Cref{prop:framework-2d}.

\begin{itemize}
\item[\bf (B1)] 
The finite element sequence
\begin{equation}\label{eq:2D:compl}
	P_1 \xrightarrow{\subset} U_h \xrightarrow[]{\hess} \bm \Sigma_h \xrightarrow[]{\operatorname{rot}} \bm Q_h\xrightarrow{}0
\end{equation} is an exact complex. The exactness also holds when restricted on any (extended) patch $\omega =  \omega_{\sigma}^{h}$ or $\omega_{\sigma}^{[m]}$  for non-negative $m$ and $\sigma \in \mathsf S$, namely,
the sequence
\begin{equation}
	P_1 \xrightarrow{\subset} U_h(\omega) \xrightarrow[]{\hess} \bm \Sigma_h(\omega) \xrightarrow[]{\operatorname{rot}} \bm Q_h(\omega) \xrightarrow{}0
\end{equation}
is exact, where $U_h(\omega), \bm \Sigma_h(\omega), \bm Q_h(\omega)$ are the restrictions  of $U_h,\bm \Sigma_h, \bm Q_h$ on $\omega$, respectively.
\item[\bf (B2)] 
For $j = 0,1,2$, it holds that $\hess \tilde \varphi_{\bm x, j} \in \Span\{\varphi_{\bm e, j'} ; \bm e \in \mathsf E, j' = 0,1,2\},$ and $\rot \tilde \varphi_{\bm e, j} \in \Span\{\varphi_{\bm f, j'} ; \bm f \in \mathsf F, j' = 0,1,2\}.$ 
\item[\bf (B3)] The collection $L_{\sigma}^0$ of the degrees of freedom  vanishes for the linear functions, namely, $L_{\sigma}^0(u) = 0$ for $u \in P_1$ and $\sigma \in \mathsf S$.
\item[\bf (B4)] For $l = 0,1,2$, the mapping $L_{\sigma}^l$ is locally defined and locally supported: for $l = 0 $ the value of $L_{\sigma}^0 u $ on $\sigma$ only depends on the value of $u \in U_h$ on $\omega_{\sigma}$, and the support of $L_{\sigma}^0 u$ and $\varphi_{\bm x, j}$ are a subset of $\omega_{\sigma}$, for $j = 0,1,2$. Similar conditions hold for $\bm \Sigma_h (l = 1)$ and $\bm Q_h (l = 2)$.
 
\end{itemize}

\begin{proposition}
    \label{prop:framework-2d}
Under Assumptions (B1)-(B4), there exist operators $\pi^l$, $l = 0,1,2$, such that $\pi^0 : H^2(\Omega) \to U_h$, $\pi^1 : H(\rot,\Omega; \mathbb S) \to \bm \Sigma_h$ and $\pi^2 : L^2(\Omega;\mathbb R^2) \to \bm Q_h$ are projection operators, and the diagram \eqref{eq:maincd-2d} commutes. 

If moreover, $\mathcal T$ is shape-regular, and the shape function space and degrees of freedom in each element are affine-interpolant equivalent to each other, then the projection operators are locally bounded, i.e., \eqref{eq:thmbound-2d} holds. As a result, all the operators are globally bounded, i.e., $\pi^0$ is $H^2$ bounded, $\pi^1$ is $H(\rot;\bS)$ bounded, and $\pi^2$ is $L^2$ bounded.
\end{proposition}
The proof of \Cref{prop:framework-2d} is similar to the argument in \Cref{sec:3d}.

\subsection{Proof of \Cref{thm:main-2d}}
\label{sec:proof-main-2d}

This subsection verifies that the finite element spaces $U_h$, $\bm \Sigma_h$ and $\bm Q_h$, introduced in \Cref{sec:dof-2d}, satisfy Assumptions (B1)--(B4). To this end, consider the following two edge bubble complexes: 
\begin{equation}
    \label{eq:edgebubble1-2d}
    0 \xrightarrow{} (\lambda_0\lambda_1)^3P_{k-4}(\bm e) \xrightarrow{\partial^2/\partial\bm t^2}(\lambda_0 \lambda_1) P_{k-2}(\bm e)/ P_1(\bm e) \xrightarrow{} 0,
\end{equation}
and 
\begin{equation}
    \label{eq:edgebubble2-2d}
     0 \xrightarrow{}  (\lambda_0\lambda_1)^2P_{k-3}(\bm e) \xrightarrow{\partial / \partial \bm t } (\lambda_0\lambda_1) P_{k-2}(\bm e) / P_0(\bm e) \xrightarrow{} 0,
\end{equation}

and the face bubble complex,
\begin{equation}
    \label{eq:2d-facebubble}
0 \xrightarrow{}(\lambda_0\lambda_1\lambda_2)^2P_{k-4}(\bm f) \xrightarrow[]{\operatorname{\hess}} B_{\bm f, k}^{\rot; \bS} \xrightarrow[]{\operatorname{rot}} P_{k-1}(\bm f; \mathbb R^2) / RT(\bm f) \xrightarrow{}0.
\end{equation}

The exactness of the edge bubble complexes \eqref{eq:edgebubble1-2d} and \eqref{eq:edgebubble2-2d} can be easily verified. The following lemma shows the exactness of \eqref{eq:2d-facebubble}.
\begin{lemma}
\label{lem:exact-2d-facebubble}
Suppose that $ k\ge 5$, then the polynomial sequence \eqref{eq:2d-facebubble} is an exact complex.
\end{lemma}
\begin{proof}
First, recall the polynomial Hessian complex in two dimensions from \cite{2002ArnoldWinther},
\begin{equation}
    \label{eq:hess-poly-2d}
    P_1 \xrightarrow{\subset} P_{k+2}(\bm f) \xrightarrow[]{\operatorname{\hess}} P_k(\bm f; \mathbb S) \xrightarrow[]{\operatorname{rot}} P_{k-1}(\bm f; \mathbb R^2) \xrightarrow{}0.
\end{equation}

Since 
$$\dim P_{k-4} + 2\dim P_{k-1} - \dim {RT} = \frac{1}{2}(k-3)(k-2) + k(k+1) - 3 = \frac{3}{2}(k^2 - k) = \dim  B_{\bm f, k}^{\rot; \bS},$$
it suffices to show that if $\bm \sigma \in  B_{\bm f, k}^{\rot; \bS}$ satisfies $\rot \bm \sigma = 0$, then there exists $u \in (\lambda_0\lambda_1\lambda_2)^2 P_{k-4}$, such that $\bm \sigma = \hess u $.

From \eqref{eq:hess-poly-2d}, there exists $u \in P_{k+2}(\bm f)$ such that $\hess u = \bm \sigma$. Moreover, assume that $u(\bm x) = \nabla u(\bm x) = 0$ for some vertex $\bm x$ of $\bm f$.
Since $\frac{\partial}{\partial \bm t} \grad u = (\hess u )\bm t = 0$ on the boundary $\partial \bm f$, it then holds that $\nabla u =0$ on $\partial \bm f$, and therefore $u = 0$ on $\partial \bm f$. Consequently, $u \in (\lambda_0\lambda_1\lambda_2)^2 P_{k-4}(\bm f).$
\end{proof}
\begin{remark}
Another approach is to show the discrete rot operator $: \bm \Sigma_h \to \bm Q_h$ is surjective, which was proved in \cite[Lemma 3.2]{2015HuZhang}.
\end{remark}

\begin{proof}[Proof of \Cref{thm:main-2d}]
The degrees of freedom imply that Assumptions (B3) and (B4) hold. For Assumption (B1), the exactness can be found in \cite{2018ChristiansenHuHu}. 
It suffices to check (B2).

Suppose that $\varphi \in U_h$ vanishes at all but the first set of the degrees of freedom \eqref{dof:1a} of $U_h$, set $\bm \sigma = \hess \varphi$. 
Clearly, it holds that $$\bm \sigma(\bm x) = \hess \varphi(\bm x) = 0 \text{ at each vertex }x \in \mathsf V,$$ hence $\bm \sigma$ vanishes at \eqref{dof:2b}. 
On edge $\bm e \in \mathsf E$, since for a given $p \in [\lambda_0 \lambda_1 P_{k-2}(\bm e)]/ P_1(\bm e)$, it follows from the exactness of \eqref{eq:edgebubble1-2d} that there exists $b \in (\lambda_0\lambda_1)^3P_{k-4}(\bm e)$ such that $\frac{\partial^2}{\partial \bm t^2} b = p$. Hence, $$(\bm t^{T} \bm \sigma \bm t,p)_{\bm e} = (\frac{\partial^2}{\partial \bm t^2} u , \frac{\partial^2}{\partial \bm t^2} b)_{\bm e} = 0,$$ indicating that $\bm \sigma$ vanishes at \eqref{dof:2c}. 
Similarly, the exactness of \eqref{eq:edgebubble2-2d} implies that $(\bm n^T \bm \sigma\bm t,p)_{\bm e} = 0$ for $p \in (\lambda_0\lambda_1)P_{k-2}(\bm e)/P_0(\bm e)$, indicating that $\bm \sigma$ vanishes at \eqref{dof:2d}.

Inside $\bm f$, since $\bm \sigma$ is rot free, it suffices to show that for all $\bm \eta \in B_{\bm f, k}^{\rot; \bS} $, it holds $(\mathcal P_{\hess  B^1_{\bm f, k-4} } \bm \sigma, \mathcal P_{\hess  B^1_{\bm f, k-4} }\bm \eta)_{\bm f} = 0$. It follows from the exactness of \eqref{eq:2d-facebubble} that $$\mathcal P_{\hess  B^1_{\bm f, k-4} } \bm \eta = \hess \psi \text{ for some }\psi \in B^1_{\bm f, k-4},$$ and therefore $$(\mathcal P_{\hess  B^1_{\bm f, k-4} } \bm \sigma, \mathcal P_{\hess  B^1_{\bm f, k-4} }\bm \eta)_{\bm f}  = (\hess \varphi, \hess \psi)_{\bm f} = 0.$$ This shows that $\bm \sigma $ vanishes at \eqref{dof:2e}. In summary, it holds that $\bm \sigma$ vanishes at all but the first set of the degrees of freedom \eqref{dof:2a}.

Suppose that $\bm \sigma \in \bm \Sigma_h$ vanishes at all but the first set of the degrees of freedom \eqref{dof:2a} of $\bm \Sigma_h$, denote by $\bm q=\rot \bm \sigma.$ It follows from the exactness of \eqref{eq:2d-facebubble} that there for any $\bm w \in P_{k-1}(\bm f; \mathbb R^2) / RT(\bm f)$ there exists $\bm \eta$ such that $$\rot \bm \eta = \bm w \text{ and }\bm \eta \perp \hess [(\lambda_0\lambda_1\lambda_2)^2P_{k-4}(\bm f)].$$ As a result, it follows from \eqref{dof:2e} that $(\bm q, \bm w)_{\bm f} = 0$, which completes the proof.
\end{proof}

\subsection{Applications to a finite element complex on the Clough--Tocher split}

Besides the finite element complex introduced in \Cref{sec:dof-2d}, another example of the finite element gradgrad complex starts from the Hsieh--Clough--Tocher element will be discussed in this subsection, whose rotated version was also discussed in \cite{2022ChristiansenHu}.

$$
P_1 \xrightarrow{\subset} U_h^{\mathrm{CT}} \xrightarrow[]{\hess} \bm \Sigma_h^{\mathrm{CT}}  \xrightarrow[]{\operatorname{rot}} \bm Q_h^{\mathrm{CT}} \xrightarrow{}0,
$$
where $U_h^{\mathrm{CT}}$,$\bm \Sigma_h^{\mathrm{CT}}$ and $\bm Q_h^{\mathrm{CT}}$ will be defined later in \eqref{eq:Ucth}, \eqref{eq:Sigmacth} and \eqref{eq:Qcth}, respectively.
In this subsection, the face $\bm f$ is split into three triangles $\bm f_{i}$, the sub-triangle with respect to vertex $\bm x_i$, $i = 0,1,2$.

The finite element complex is illustrated as follows.

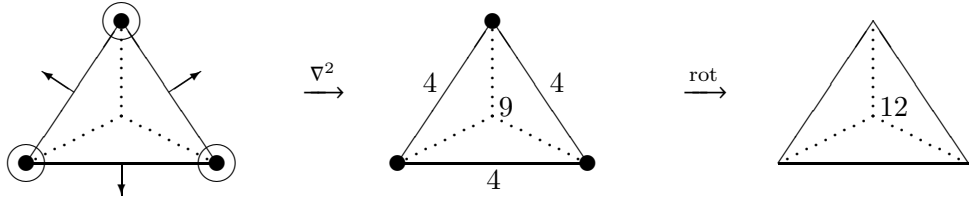
\begin{figure}[htb] \begin{center} \setlength{\unitlength}{1.20pt}

    \begin{picture}(270,45)(0,5)
\put(0,0){
    \begin{picture}(70,70)(0,0)
    \put(0,0){\line(1,0){60}}
    \put(0,0){\line(2,3){30}}
    \put(60,0){\line(-2,3){30}}
    \multiput(30,15)(0,3){10}{\circle*{1}}
    \multiput(30,15)(-3,-1.5){10}{\circle*{1}}
    \multiput(30,15)(3,-1.5){10}{\circle*{1}}
    \put(0,0){\circle*{5}}\put(0,0){\circle{12}}
    \put(60,0){\circle*{5}}\put(60,0){\circle{12}}
    \put(30,45){\circle*{5}}\put(30,45){\circle{12}}
   \put(15,22.5){\vector(-3,2){10}}
   \put(45,22.5){\vector(3,2){10}}
    \put(30,0){\vector(0,-1){10}}
   \end{picture}}
  
\put(90,20){$\xrightarrow[]{\hess}$}
\put(120,0){\begin{picture}(70,70)(0,0)\put(0,0){\line(1,0){60}}
    \put(0,0){\line(2,3){30}}\put(60,0){\line(-2,3){30}}
    \put(0,0){\circle*{5}}
    \put(60,0){\circle*{5}}
    \put(30,45){\circle*{5}}
    \multiput(30,15)(0,3){10}{\circle*{1}}
    \multiput(30,15)(-3,-1.5){10}{\circle*{1}}
    \multiput(30,15)(3,-1.5){10}{\circle*{1}}
    \put(8,22.5){$4$}
    \put(48,22.5){$4$}
    \put(28,-8){$4$}
    \put(32,15){$9$}
   \end{picture}}
   \put(210,20){$\xrightarrow[]{\rot}$}
  \put(240,0){\begin{picture}(70,70)(0,0)\put(0,0){\line(1,0){60}}
    \put(0,0){\line(2,3){30}}\put(60,0){\line(-2,3){30}}
    \multiput(30,15)(0,3){10}{\circle*{1}}
    \multiput(30,15)(-3,-1.5){10}{\circle*{1}}
    \multiput(30,15)(3,-1.5){10}{\circle*{1}}
    \put(32,15){$12$}
   \end{picture}}
\end{picture}
\end{center}
\end{figure}

\subsubsection*{$H^2$ conforming finite element space} The shape function $U^{\mathrm{CT}}(\bm f)$ is chosen as 
$$ U^{\mathrm{CT}}(\bm f) := \{ u \in C^1(\bm f) ~:~ u|_{\bm f_i} \in P_3(\bm f_i)\},$$
whose dimension is 12.
For $ u \in U^{\mathrm{CT}}(\bm f)$, define the degrees of freedom as:
\begin{enumerate}[label=(8\alph*).,ref=8\alph*]
\item the function value and first order derivatives $u(\bm x)$, $\frac{\partial}{\partial x}u(\bm x)$ and $\frac{\partial}{\partial y}u(\bm x)$ at each vertex $\bm x$ of $\bm f$;
\item the moment of the normal-tangential derivative $(\frac{\partial^2}{\partial \bm n \partial \bm t} u, \frac{\partial}{\partial \bm t} (\lambda_0\lambda_1))_{\bm e},$ where $\lambda_0$ and $\lambda_1$ are the barycenter coordinates of $\bm e$. 
\end{enumerate}

The above degrees of freedom are unisolvent, and the resulting finite element space is the HCT element \cite{1965CloughTocher},
\begin{equation} \label{eq:Ucth} U_{h}^{\mathrm{CT}} := \{ u \in C^1(\Omega) : u|_{\bm f} \in U^{\mathrm{CT}}(\bm f), \forall \bm f \in \mathcal T ; u \text{ is } C^1 \text{ at each vertex}\}.\end{equation}

\subsubsection*{$H(\rot; \bS)$ conforming finite element space}  
The shape function space of the $H(\rot; \mathbb S)$ conforming space is chosen as 
$$ \bm \Sigma^{\mathrm{CT}}(\bm f) := \{ \bm \sigma \in H(\rot, \bm f;\mathbb S) ~:~ \bm \sigma|_{\bm f_i} \in P_1(\bm f_i;\mathbb S)\}, $$
whose dimension is 15.
For $\bm \sigma \in \bm \Sigma^{\mathrm{CT}}(\bm f)$, the degrees of freedom as defined as:
\begin{enumerate}[label=(9\alph*).,ref=9\alph*]
    \item the moments $(\bm \sigma\bm t, \bm p)_{\bm e}$ of its tangential component, for $\bm p \in RT(\bm e)$, on each edge $\bm e$ of $\bm K$;
    \item the moment of the normal-tangential component, $(\bm n^T \bm \sigma \bm t, \frac{\partial}{\partial \bm t} (\lambda_0\lambda_1))_{\bm e};$
    \item the moments inside $\bm f$, $(\rot \bm \sigma, \rot \bm p)_{\bm f}$ for $\bm p \in B_{\bm f, \mathrm{CT}}$. Here, the face bubble space 
    $$ B_{\bm f, \mathrm{CT}} := \{ \bm \sigma \in \bm \Sigma^{\mathrm{CT}}(\bm f): \bm \sigma \bm t = 0 \text{ on edge } \bm e \text{ of } \bm f\},$$
    whose dimension is 3.
\end{enumerate}

The above degrees of freedom are unisolvent, and the resulting finite element space is a rotation of the Johnson--Mercier element \cite{1978JohnsonMercier},
\begin{equation} \label{eq:Sigmacth}\bm \Sigma_{h}^{\mathrm{CT}} := \{ \bm \Sigma \in H(\rot,\Omega;\mathbb S) : u|_{\bm f} \in \bm \Sigma^{\mathrm{CT}}(\bm f), \forall \bm f \in \mathcal T\}.\end{equation}

\subsubsection*{$L^2(\mathbb R^2)$ finite element space} The shape function of the $L^2(\mathbb R^2)$ finite element space is taken as 
$$ \bm Q^{\mathrm{CT}}(\bm f) := \{ \bm q \in L^2(\bm f) ~:~ \bm q|_{\bm f_i} \in P_0(\bm f)\}, $$
whose dimension is 6.
For $\bm q \in \bm Q^{\mathrm{CT}}(\bm f),$ the degrees of freedom are defined as :
\begin{enumerate}[label=(10\alph*).,ref=10\alph*]
     \item the moments $(\bm v, \bm w)_{\bm f}$ for $\bm w \in RT(\bm f)$;
\item the moments $(\bm v, \bm z)_{\bm f}$ for $\bm z \in \bm Q^{\mathrm{CT}}(\bm f) /{RT(\bm f)}$.
\end{enumerate}
The global space is defined as 
\begin{equation}\label{eq:Qcth}
Q_{h}^{\textrm{CT}} := \{ \bm q \in L^2(\Omega; \mathbb R^2) ~:~ \bm q|_{\bm f} \in \bm Q^{\mathrm{CT}}(\bm f), \forall \bm f \in \mathcal T \}. 
 \end{equation}
For this complex, Assumptions (B1)--(B4) can be similarly verified.
\bibliographystyle{plain}
\bibliography{reference}
\end{document}

%% file: tikzfile/omega_x.tikz
\begin{tikzpicture}[scale = .6]
		\node (0) at (0, 0) {~~$x$};
		\node  (1) at (0.75, 1) {};
		\node  (2) at (-1, 1) {};
		\node (3) at (-1, -0.5) {};
		\node  (4) at (0.25, -1.5) {};
		\node  (5) at (1.5, -0.75) {};
		\node  (6) at (1.25, 2.25) {};
		\node  (7) at (0, 2.25) {};
		\node  (8) at (-1.5, 2) {};
		\node  (9) at (-2.75, 0.5) {};
		\node  (10) at (-2, -1) {};
		\node  (11) at (-1.25, -2.25) {};
		\node  (12) at (0.25, -3) {};
		\node  (13) at (2, -2.75) {};
		\node  (14) at (3, -2) {};
		\node  (15) at (3.25, -1) {};
		\node  (16) at (2.5, 1.25) {};
		\node  (17) at (2.75, 5.25) {};
		\fill[green]  (1.center) to (2.center) to (3.center) to (4.center) to (5.center) to  (1.center);
		\draw (0.center) to (2.center);
		\draw (2.center) to (3.center);
		\draw (3.center) to (0.center);
		\draw (0.center) to (4.center);
		\draw (3.center) to (4.center);
		\draw (4.center) to (5.center);
		\draw (5.center) to (0.center);
		\draw (0.center) to (1.center);
		\draw (1.center) to (5.center);
		\draw (2.center) to (1.center);
		\draw (1.center) to (6.center);
		\draw (7.center) to (6.center);
		\draw (7.center) to (1.center);
		\draw (2.center) to (7.center);
		\draw (2.center) to (8.center);
		\draw (8.center) to (7.center);
		\draw (6.center) to (16.center);
		\draw (1.center) to (16.center);
		\draw (16.center) to (5.center);
		\draw (5.center) to (15.center);
		\draw (15.center) to (16.center);
		\draw (15.center) to (13.center);
		\draw (13.center) to (5.center);
		\draw (12.center) to (13.center);
		\draw (12.center) to (4.center);
		\draw (5.center) to (12.center);
		\draw (4.center) to (11.center);
		\draw (11.center) to (12.center);
		\draw (10.center) to (11.center);
		\draw (10.center) to (3.center);
		\draw (3.center) to (11.center);
		\draw (3.center) to (9.center);
		\draw (9.center) to (10.center);
		\draw (9.center) to (2.center);
		\draw (8.center) to (9.center);
\end{tikzpicture}

%% file: tikzfile/omega_e.tikz
\begin{tikzpicture}[scale = .6]
		\node  (0) at (0, 1) {};
		\node  (1) at (0, -1.25) {};
		\node  (2) at (-1, 0) {};
		\node  (3) at (1.5, 0) {};
		\node  (4) at (1.5, 2) {};
		\node  (5) at (3, 0.5) {};
		\node  (6) at (2.75, -1.5) {};
		\node  (7) at (1.25, -2.5) {};
		\node  (8) at (-1.25, -2.5) {};
		\node  (9) at (-2.5, -0.75) {};
		\node  (10) at (-2, 1.75) {};
		\node  (11) at (-0.75, 3) {};
		\fill[green]  (0.center) to (2.center) to (1.center) to (3.center) to (0.center);
		\draw (0.center) to (1.center);
		\draw (1.center) to (2.center);
		\draw (2.center) to (0.center);
		\draw (0.center) to (3.center);
		\draw (3.center) to (1.center);
		\draw (11.center) to (0.center);
		\draw (0.center) to (4.center);
		\draw (4.center) to (3.center);
		\draw (3.center) to (5.center);
		\draw (4.center) to (5.center);
		\draw (5.center) to (6.center);
		\draw (6.center) to (7.center);
		\draw (3.center) to (6.center);
		\draw (1.center) to (7.center);
		\draw (1.center) to (6.center);
		\draw (1.center) to (8.center);
		\draw (8.center) to (7.center);
		\draw (2.center) to (8.center);
		\draw (9.center) to (8.center);
		\draw (9.center) to (2.center);
		\draw (11.center) to (2.center);
		\draw (11.center) to (10.center);
		\draw (10.center) to (2.center);
		\draw (11.center) to (4.center);
		\draw (10.center) to (9.center);
\end{tikzpicture}

%% file: tikzfile/omega_f.tikz
\begin{tikzpicture}[scale=.4]
		\node  (0) at (0, 2) {};
		\node  (1) at (-1.5, -1.25) {};
		\node  (2) at (1.75, -1.25) {};
		\node  (3) at (2.25, 2.75) {};
		\node  (4) at (3.75, 0.75) {};
		\node  (5) at (3.75, -2.5) {};
		\node  (6) at (1.5, -3.75) {};
		\node  (7) at (-1.25, -3.5) {};
		\node  (8) at (-3.75, -0.75) {};
		\node  (9) at (-3, 1.5) {};
		\node  (10) at (-1, 3.5) {};
		\fill[green] (0.center) to (1.center) to (2.center) to (0.center);
		\draw (0.center) to (1.center);
		\draw (1.center) to (2.center);
		\draw (2.center) to (0.center);
		\draw (0.center) to (3.center);
		\draw (3.center) to (4.center);
		\draw (0.center) to (4.center);
		\draw (4.center) to (2.center);
		\draw (2.center) to (5.center);
		\draw (4.center) to (5.center);
		\draw (5.center) to (6.center);
		\draw (2.center) to (6.center);
		\draw (6.center) to (7.center);
		\draw (7.center) to (2.center);
		\draw (1.center) to (7.center);
		\draw (7.center) to (8.center);
		\draw (8.center) to (1.center);
		\draw (1.center) to (9.center);
		\draw (9.center) to (8.center);
		\draw (9.center) to (0.center);
		\draw (10.center) to (0.center);
		\draw (3.center) to (10.center);
		\draw (10.center) to (9.center);
\end{tikzpicture}

%% file: interp2.bbl
\begin{thebibliography}{10}

\bibitem{1968Argyris}
John~H Argyris, Isaac Fried, and Dieter~W Scharpf.
\newblock The {TUBA} family of plate elements for the matrix displacement
  method.
\newblock {\em The Aeronautical Journal}, 72(692):701--709, 1968.

\bibitem{2021ArnoldGuzman}
Douglas Arnold and Johnny Guzm\'{a}n.
\newblock Local {$L^2$}-bounded commuting projections in {FEEC}.
\newblock {\em ESAIM Math. Model. Numer. Anal.}, 55(5):2169--2184, 2021.

\bibitem{2018Arnold}
Douglas~N. Arnold.
\newblock {\em Finite element exterior calculus}, volume~93 of {\em CBMS-NSF
  Regional Conference Series in Applied Mathematics}.
\newblock Society for Industrial and Applied Mathematics (SIAM), Philadelphia,
  PA, 2018.

\bibitem{2006ArnoldFalkWinther}
Douglas~N. Arnold, Richard~S. Falk, and Ragnar Winther.
\newblock Finite element exterior calculus, homological techniques, and
  applications.
\newblock {\em Acta Numer.}, 15:1--155, 2006.

\bibitem{2010ArnoldFalkWinther}
Douglas~N. Arnold, Richard~S. Falk, and Ragnar Winther.
\newblock Finite element exterior calculus: from {H}odge theory to numerical
  stability.
\newblock {\em Bull. Amer. Math. Soc. (N.S.)}, 47(2):281--354, 2010.

\bibitem{2002ArnoldWinther}
Douglas~N. Arnold and Ragnar Winther.
\newblock Mixed finite elements for elasticity.
\newblock {\em Numer. Math.}, 92(3):401--419, 2002.

\bibitem{2018ChristiansenHuHu}
Snorre~H. Christiansen, Jun Hu, and Kaibo Hu.
\newblock Nodal finite element de {R}ham complexes.
\newblock {\em Numer. Math.}, 139(2):411--446, 2018.

\bibitem{2022ChristiansenHu}
Snorre~H Christiansen and Kaibo Hu.
\newblock Finite element systems for vector bundles: elasticity and curvature.
\newblock {\em Foundations of Computational Mathematics}, pages 1--52, 2022.

\bibitem{1965CloughTocher}
RW~Clough and JL~Tocher.
\newblock Finite element stiffness matricesfor analysis of plates in bending.
\newblock In {\em Proc. Conf. on Matrix Methodsin Structural Mechanics,
  Wright-Patterson AFB, Ohio}, 1965.

\bibitem{2014FakWinther}
Richard~S. Falk and Ragnar Winther.
\newblock Local bounded cochain projections.
\newblock {\em Math. Comp.}, 83(290):2631--2656, 2014.

\bibitem{2015FalkWinther}
Richard~S. Falk and Ragnar Winther.
\newblock Double complexes and local cochain projections.
\newblock {\em Numer. Methods Partial Differential Equations}, 31(2):541--551,
  2015.

\bibitem{2021HuLiang}
Jun Hu and Yizhou Liang.
\newblock Conforming discrete gradgrad-complexes in three dimensions.
\newblock {\em Mathematics of Computation}, 90(330):1637--1662, 2021.

\bibitem{2023HuLiangLin}
Jun Hu, Yizhou Liang, and Ting Lin.
\newblock Local bounded commuting projection operator for discrete de {R}ham
  complexes.
\newblock {\em arXiv preprint arXiv:2303.09359}, 2023.

\bibitem{2015HuZhang}
Jun Hu and Shangyou Zhang.
\newblock A family of conforming mixed finite elements for linear elasticity on
  triangular grids.
\newblock {\em arXiv preprint arXiv:1406.7457}, 2014.

\bibitem{1978JohnsonMercier}
Claes Johnson and Bertrand Mercier.
\newblock Some equilibrium finite element methods for two-dimensional
  elasticity problems.
\newblock {\em Numerische Mathematik}, 30:103--116, 1978.

\bibitem{2022PaulySchomburg}
Dirk Pauly and Michael Schomburg.
\newblock Hilbert complexes with mixed boundary conditions--part 3: Biharmonic
  complexes.
\newblock {\em arXiv preprint arXiv:2207.11778}, 2022.

\bibitem{2020PaulyZulehner}
Dirk Pauly and Walter Zulehner.
\newblock The div{D}iv-complex and applications to biharmonic equations.
\newblock {\em Appl. Anal.}, 99(9):1579--1630, 2020.

\bibitem{2015Quenneville}
Vincent Quenneville-Belair.
\newblock {\em A {N}ew {A}pproach to {F}inite {E}lement {S}imulations of
  {G}eneral {R}elativity}.
\newblock ProQuest LLC, Ann Arbor, MI, 2015.
\newblock Thesis (Ph.D.)--University of Minnesota.

\bibitem{2009Zhang}
Shangyou Zhang.
\newblock A family of 3d continuously differentiable finite elements on
  tetrahedral grids.
\newblock {\em Applied Numerical Mathematics}, 59(1):219--233, 2009.

\end{thebibliography}
